\documentclass[%
 reprint,
nofootinbib,
 amsmath,amssymb,
 aps,
 onecolumn,
pra,
]{revtex4-2}

\usepackage{graphicx}
\usepackage{dcolumn}
\usepackage{bm}
\usepackage{amssymb}
\usepackage{mathtools}
\usepackage{bm}
\usepackage{url}
\usepackage{hyperref}
\usepackage{amsthm}
\usepackage{algpseudocode}
\usepackage{algorithm}
\usepackage{xcolor}
\usepackage{cleveref}

\DeclareMathOperator{\R}{\mathbb{R}}

\newcommand{\norm}[1]{\left\lVert#1\right\rVert}

\newtheorem{theorem}{Theorem}[section]

\newtheorem{lemma}[theorem]{Lemma}

\begin{document}

\preprint{APS/123-QED}

\title{Computing an Aircraft's Gliding Range and Minimal Return Altitude in Presence of Obstacles and Wind}

\author{Giovanni Piccioli}
\affiliation{École polytechnique fédérale de Lausanne (EPFL) \\ Statistical Physics of Computation Laboratory}%
\email{giovannipiccioli@gmail.com}

\date{\today}

\begin{abstract}
In the event of a total loss of thrust, a pilot must identify a reachable landing site and subsequently execute a forced landing. To do so, they must estimate which region on the ground can be reached safely in gliding flight. We call this the gliding reachable region (GRR). To compute the GRR, we employ an optimal control formulation aiming to reach a point in space while minimizing altitude loss. A simplified model of the aircraft's dynamics is used, where the effect of turns is neglected.
The resulting equations are discretized on a grid and solved numerically. Our algorithm for computing the GRR is fast enough to run in real time during flight, it accounts for ground obstacles and wind, and for each point in the GRR it outputs the path to reach it with minimal loss of altitude.

A related problem is estimating the minimal altitude an aircraft needs in order to glide to a given airfield in the presence of obstacles. This information enables pilots to plan routes that always have an airport within gliding distance. We formalize this problem using an optimal control formulation based on the same aircraft dynamics model. The resulting equations are solved with a second algorithm that outputs the minimal re-entry altitude and the paths to reach the airfield from any position while avoiding obstacles.

The algorithms we develop are based on the Ordered Upwind Method \cite{sethian2003ordered} and the Fast Marching Method \cite{sethian1996fast}.
\end{abstract}
\maketitle

\section{Introduction}
\subsection{Motivation}
The gliding reachable region (GRR) is the region on the ground that can be safely reached in gliding flight from a certain position and altitude. Knowing the GRR allows pilots to quickly assess which landing sites are within reach in the event of an engine failure. 
To aid in this task several avionics companies such as Garmin, Foreflight, and LXNAV have developed respectively the Glide Range Ring \cite{garmin}, Glide Advisor \cite{foreflight}, and glide range area \cite{lxnav} aids. These algorithms all display the GRR on the navigation map in real time, allowing the pilot to quickly determine the reachability of a landing site. Aside from emergencies, estimating the reachable region is also of great importance for glider pilots.
In this work, we consider two problems:
\begin{enumerate}
    \item \textbf{Gliding Reachable Region Problem (GRRP):} given the initial position and altitude of an aircraft, what is the GRR? For each point in the GRR, what is the path the aircraft should take to reach it while minimizing the loss of altitude?
    \item \textbf{Minimal Return Altitude Problem (MRAP):} given an airfield and a position $\bm x\in \R^2$ on a map, what is the minimal altitude an aircraft must be over $\bm x$ in order to reach the airfield in gliding flight? Supposing the aircraft is at position $\bm x$ at the minimal altitude, what path should it take to reach the airport?
\end{enumerate}
We already commented on the relevance of the first problem, we now comment on why the second one is also interesting.
Suppose we want to know if we can reach a certain airport in gliding flight starting from position $\bm x\in \R^2$ and altitude $z$. One option is to compute the GRR and check if the airport falls within it. However, finding the GRR can be computationally expensive. Suppose instead we have access to a database that for each airfield $q$ and position $\bm y\in\R^2$ stores  $A_q(\bm y)$, the minimal altitude to return to the airfield from $\bm y$. Then to establish if airfield $q$ is reachable it's sufficient to check if $z\geq A_q(\bm x)$. In principle the contour lines of $A_q(\cdot)$ can be printed onto a paper map, allowing the pilot to know if they're within gliding range of the airfield without resorting to electronic devices. In addition, knowing the minimal return altitudes in advance can help pilots to plan flights and maintain an altitude sufficient to always reach a landable location.

GRRP and MRAP appear to be two quite different problems. We decide to study them in the same article because we employ similar mathematical formalisms and algorithms to solve them.

\subsection{Our contribution}
We develop the algorithms Glikonal-G and Glikonal-M to solve respectively GRRP and MRAP. The algorithms are respectively based on a variant of the Ordered Upwind Method (OUM) \cite{sethian2003ordered} and Fast Marching Method \cite{sethian1996fast} (FMM). We claim Glikonal-G is fast enough to run on an on-board computer in real time, taking less than $0.2$ seconds to run on a $100\times100$ grid. Also given a point within the GRR, Glikonal-G can compute the optimal gliding path to reach it. Similarly, Glikonal-M can produce a feasible (but not necessarily optimal, in the sense of minimal altitude loss) path that leads to the airfield.

First, we model the aircraft's dynamics using a 3DOF model, where the airspeed vector is the control parameter. We assume in particular that the aircraft is able to take arbitrarily sharp turns. Also, the sink rate of the aircraft is assumed to depend exclusively on the airspeed, thus we do not take into account the fact that an airplane experiences a more negative sink rate when taking turns. In Appendix \ref{app:altitude_loss_turns} we comment on the validity of this approximation and on methods to overcome it. 

GRRP is then formulated as an optimal control problem consisting of determining the trajectory from the initial position of the aircraft to any point in space that minimizes the altitude loss, while taking wind and ground obstacles into account. If this trajectory exists, then the final point is part of the GRR. The Hamilton-Jacobi-Bellman (HJB) PDE for this problem is derived and the Glikonal-G algorithm is introduced to solve it on a discrete grid. 

To solve MRAP instead one constrains the final position of the aircraft to the the airport position and the final altitude to be a safe altitude above the airport. Then a related optimal control problem is formulated, giving an HJB equation similar to that of GRRP. 
The Glikonal-M algorithm solves the HJB PDE on a discrete grid outputting the minimal altitude function.
\subsection{Related literature}
The simplest method to compute the GRR in the absence of obstacles is to assume that the aircraft will follow a path composed by an initial turn towards the desired heading and then a straight path at the best glide speed. This was explored in \cite{di2016optimizing,chen2024gliding} where the turn with the least altitude loss is computed, and in \cite{atkins2006emergency} where the effect of wind is also considered (although without optimizing the turn).

In other works the task of finding a gliding path between two points, where the initial and final headings are also given, is considered. In \cite{ambrosino2009path, eng2010automating,atkins2006emergency, stephan2016fast,izuta2017path, coombes2014reachability} paths are allowed to be more complex, by concatenating basic maneuvers (e.g., constant bank turns, steady straight flight, accelerated straight flight) in a sequence. In none of these works, however, obstacles are considered. Most of these studies rely on a generalization of Dubins paths \cite{dubins1957curves}, i.e., paths whose maximal curvature is bounded. 

In \cite{coombes2014reachability} the effect of a uniform wind is considered.
In \cite{adler2012optimal,akametalu2018reachability} the idea of defining a path between two points that minimizes the lost altitude by concatenating basic maneuvers was extended to take ground obstacles and wind into account. In both these works a 6DOF model of the aircraft is used. The algorithm proposed in \cite{adler2012optimal} consists of searching over the discretized state space in six dimensions.

In other studies, the optimal control formulation of the problem is also considered. In particular \cite{akametalu2018reachability,bayen2007aircraft} use HJB equations to conduct the reachability analysis. However \cite{bayen2007aircraft} considers only the longitudinal movements of the aircraft (i.e., no turns are allowed).
Another line of work proposed to deal with obstacles by means of a visibility graph \cite{meuleau2009emergency,meuleau2009comparison}. The obstacles are modeled as polygons: in this case, the shortest path between two points on the plane, while avoiding the polygons, will touch the polygons exclusively on the vertices. This remark is exploited to build piecewise linear paths that avoid obstacles. 
The work closest to ours is \cite{segal2023altitude}, in which a similar aircraft dynamics to ours is used and obstacles of arbitrary shape are considered. The algorithm in \cite{segal2023altitude} relies on a modified visibility graph to compute the paths with the minimal altitude loss. Glikonal-G does fundamentally the same thing, while being more efficient and accounting for wind.
Other approaches to the problem of finding gliding paths between two points rely on genetic algorithms \cite{silva2017heuristic}, Pseudo-Spectral-Methods \cite{benasher2010pseudo}, dedicated optimization schemes \cite{fang2019emergency}.

Except for \cite{di2016optimizing,atkins2006emergency,chen2024gliding} all other algorithms allow to compute a gliding path between an initial and a single final point. Therefore to compute the GRR one would need to run the algorithm multiple times changing the final point each time, which is unpractical. In contrast, Glikonal-G computes all the trajectories to reach points in the GRR in a single run.

Turning to MRAP, the idea of precomputing the gliding trajectories from any point in the planned flight route to a safe landing zone has been explored in \cite{ayhan2019preflight}. This idea is similar to MRAP, which produces a precomputed database to determine the reachable airfields and the trajectories to get there.

This section would be incomplete without a review of the algorithms developed by Foreflight, LXNAV and Garmin, which are used in practice. While their software is proprietary and hence not published we can speculate on how it works by looking at their products. It appears that Foreflight and Garmin both use a variant of \cite{di2016optimizing} which compute the GRR by projecting straight lines in all directions from the aircraft and checking when they intersect the ground. We refer to this class of algorithms as line-of-sight methods, since it assumes that (after possibly an initial turn to the desired heading) the aircraft proceeds in a straight line. LXNAV seems to use a more sophisticated method which is able to pass 'around' the obstacles instead.

Finally, let us comment on the origin of the algorithms we use to solve the problem. Sethian first developed the FMM to solve the eikonal equation in \cite{sethian1996fast}. The eikonal equation describes the motion of a propagating front in a medium. As we will see, the anisotropic version of the eikonal equation, belonging to the class of static HJB equations \cite{bardi1997optimal,bressan2011viscosity}, is linked to the optimal dynamics of the aircraft in the presence of wind. A set of algorithmic schemes called Ordered Upwind Methods (OUM) able to solve this class of equations has been developed by Sethian and Vladimirsky in \cite{sethian1999fast,sethian2003ordered}. Glikonal-G is a modification of the OUM, instead, Glikonal-M is a modified FMM.
\section{Problem definition}
\subsection{Aircraft dynamics}
In this section, we introduce the dynamical model for the glider and define the optimal control problem. Throughout the paper, we assume all reference systems are East-North-up oriented.
Let $\bm r=(\bm x,z)$ indicate the position of the glider in 3D space, where $\bm x=(x_1,x_2)$ represents the horizontal position, and $z$ is the altitude above mean sea level. 
Consider also a wind vector field $\bm W:\R^3\to \R^2$. $\bm W(\bm r)$ is a two-component vector indicating the wind (i.e., the horizontal component of the air velocity with respect to the ground) at position $\bm r$. We model the dynamics of the glider in the following way: 
\begin{align}
    \label{eq:dyn1}
    \dot{\bm x}(t)=\Tilde{\bm a}(t)+\bm W(\bm r(t)), \qquad \dot{z}(t)=s(\norm{\Tilde{\bm a}(t)})
\end{align}
Where $\Tilde{\bm a}(t)$ is the horizontal component of the airspeed at time $t$, and $s: \R_+\to (-\infty,0)$ is the sink rate as a function of the horizontal airspeed. \footnote{We can also model the vertical movements of the air with a function $s_\text{air}:\R^3\to\R$. $s_\text{air}(\bm r)$ is the vertical velocity of the air mass at position $\bm r$. For technical reasons the overall sink rate must be negative, otherwise the glider can climb up by staying put on a particular point. To take vertical air movements into account, just replace $s\to s+s_\text{air}$ in the rest of the derivation.} In straight, steady-state gliding flight, each horizontal airspeed corresponds to a vertical airspeed (or sink rate). The functional form of $s$ is given by the polar curve of the aircraft.

The vector $\bm V_\text{air}\coloneqq\left(\Tilde a_1,\Tilde a_2,s\left(\sqrt{\Tilde a_1^2+\Tilde a_2^2}\right)\right)$ is therefore the airspeed.

The control parameter is $\Tilde{\bm a}:\R_+\mapsto \R^2 $. The only assumption we make about $\Tilde{\bm a}$ is that it is a measurable function. This implies that we allow the glider's speed to change discontinuously both in direction and magnitude. $\Tilde{\bm a}$ is subject to the constraint $V_\text{S}<\norm{\bm V_\text{air}}<V_\text{NE}$, where $V_\text{S}$ and $V_\text{NE}$ are respectively the stall and never exceed speeds of the aircraft.
Let us make some additional remarks about the limitations of our dynamical model. Since the sink rate depends exclusively on the magnitude of the speed and not on the acceleration, the model neglects the effect of turns. In reality the sink rate is higher (more negative) during a turn. Also, arbitrarily sharp turns are allowed under this model. 

Define the ground speed to be $\bm a(t)=\Tilde{\bm a}(t)+ \bm W(\bm r(t))$. For mathematical convenience, we suppose that our control parameter is directly the ground speed $\bm a$. Reformulating the equations of motion in terms of $\bm a$ we obtain
\begin{align}
    \label{eq:dyn2}
    \dot{\bm x}(t)=\bm a(t), \qquad \dot{z}(t)=s(\norm{\bm a(t)-\bm W(\bm r(t))}).
\end{align}

To take obstacles into account we say $h_\text{min}(\bm y)$ is the minimum allowed altitude at point $\bm y$.  The function $h_\text{min}:\R^2\mapsto \R$ can for example be taken to be the terrain elevation profile plus a margin. 
Finally let us specify the allowed controls: let $\mathcal{A}_{gs}(\bm r)\coloneqq \left\{\bm a \text{ s.t. } V_\text{S}<\norm{\bm V_\text{air}}<V_\text{NE}\right\} \subset \R^2$ be the set of allowed ground speeds at point $\bm r$.
Now that the aircraft dynamics and the constraints are specified we formulate GRRP and MRAP.

\subsection{Gliding Reachable Region Problem}
\label{sec:GRRP_theory}
Given an initial position and altitude $(\bm x_0,z_0)$ we wish to compute the region the aircraft can reach in gliding flight.
To do so we compute the highest altitude at which the glider can reach any given position $\bm y\in\R^2$. We define $z_0-U(\bm y)$ to be the highest altitude at which the glider can be over $\bm y$. Hence $U(\bm y)$ represents the minimal loss of altitude to reach $\bm y$ from $\bm x_0$. Mathematically $U(\bm y)$ is obtained by solving the following optimal control problem
\begin{align}
\label{eq:optimal_control_t}
    U(\bm y)&=\min_{\bm a\in \mathcal F}  \,z_0-z(t_f)\\
    \label{eq:glid_constraints}
    \text{subject to: the dynamics \eqref{eq:dyn2}}, \,\bm x(t_f)=\bm y,\; \bm x(0)=\bm x_0, \;\;&z(t)\geq h_\text{min}(\bm x(t)) \,\,\forall t\in[0,t_f], \,\, \bm a(t)\in \mathcal A_{gs}(\bm r(t)) \,\,\forall t\in[0,t_f],
\end{align}
where $\mathcal F$ is the set of measurable ground speed functions. 
If no path satisfying the constraints \eqref{eq:glid_constraints} exists, then $\bm y$ cannot be reached in gliding flight and we say $U(\bm y)$ is not defined. The GRR is then the set of points where $U(\bm y)$ is defined, in formulas
\begin{equation}
\label{eq:GRR}
    \text{GRR}(\bm x_0,z_0)\coloneqq \left\{\bm y\in\R^2 \text{ s.t.  the function } U(\bm y) \text{ satisfying \eqref{eq:optimal_control_t},\eqref{eq:glid_constraints} is defined}\right\}\,.
\end{equation}

We now aim to write an optimal control problem equivalent to \eqref{eq:optimal_control_t} but where time is replaced with the lost altitude. First remark that since $s<0$ the altitude is a decreasing function of time.
Defining the lost altitude $h(t)=z_0-z(t)$ we can we can reparametrize the dynamics using $h$ in place of $t$. With a change of variable, we obtain 
\begin{equation}
\label{eq:dyn_h}
    \frac{d\bm x}{dh} (h)=\frac{d\bm x}{dt}(t(h))\frac{dt}{dh}(h)=\frac{\bm a}{\left|s(\norm{\bm a(h)-\bm W(\bm x(h),z_0-h)})\right|}=\Tilde g(\bm x(h),z_0-h,\bm a(h)) \frac{\bm a}{\norm{\bm a}}\,.
\end{equation}
In the second to last step we renamed $\bm a(t(h))\to \bm a(h)$ and $ \bm W(\bm r(t(h)))\to \bm W(\bm x(h),z_0-h)$. In the last step, we introduced function $\Tilde g(\bm x,z,\bm a)\coloneqq\norm{\bm a}/|s(\norm{\bm a-\bm W(\bm x,z)})|$ representing the glide ratio of the aircraft: it tells us how many meters the aircraft moves horizontally for each meter of lost altitude. Notice $\Tilde g$ depends on the position of the aircraft and on its ground speed. When the wind is null, $\Tilde g$ depends only on $\norm{\bm a}$. As a result of this change of variable, we obtained a new formulation of the optimal control problem, where we got rid of the variable $z$. The reparametrized problem is
\begin{align}
\label{eq:optimal_control_h}
    U(\bm y)=\min_{\bm a\in \mathcal F}  h_f \\
    \text{subject to: eq. \eqref{eq:dyn_h}}, \,\bm x(0)=\bm x_0, \,\bm x(h_f)=\bm y,\, z_0-h\geq h_\text{min}(\bm x(h)) &\,\forall h\in [0,h_f],\,\bm a(h)\in \mathcal A_{gs}(\bm r(h)) \,\forall h\in[0,h_f],
\end{align}
where $\bm r(h)=(\bm x(h),z_0-h)$.
At this point, we can derive the HJB equation corresponding to this problem. Let $\bm x:[0,h_f]\mapsto \R^2$ be the trajectory that minimizes the lost altitude when starting from point $\bm x(0)=x_0, \, z(0)=z_0$ and arriving in $\bm x(h_f)=\bm y, z(h_f)=z_0-h_f$. 
Then we can write
\begin{align}
\label{eq:bellman_opt_int_grrp}
    U(\bm y)= \min_{\bm a\in \mathcal F} \int_0^{h_f} dh=U(\bm x(h_f-\epsilon))+\min_{\bm a\in \mathcal F}  \int_{h_f-\epsilon}^{h_f} dh= U(\bm y)+\min_{\bm a\in \mathcal{A}_{gs}(\bm r(h_f))}\left(-\nabla U(\bm y)\cdot \frac{d\bm x}{dh} (h_f) +1\right)\epsilon+ O(\epsilon^2)\,,
\end{align}
where in the last step we expanded $U(\bm x(h_f-\epsilon))=U(\bm y)-\nabla U(\bm y)\cdot\frac{d\bm x}{dh} (h_f)$.
Since this must hold for any small $\epsilon$,  the term multiplying $\epsilon$ must be zero giving the HJB equation
\begin{equation}
    \label{eq:HJB_R2}
    \min_{\bm a\in\mathcal{A}_{gs}(\bm r(h_f))} (-\nabla U(\bm y))\cdot \Tilde g(\bm y,z_0-h_f,\bm a) \frac{\bm a}{\norm{\bm a}}=-1
\end{equation}
At this point, we can further simplify the equation by first changing the minimum into a maximum using $\min_x f(x)=-\max_x -f(x)$, and then carrying out the optimization with respect to $\norm{\bm a}$, and be left therefore with optimizing only with respect to the direction $\bm{\hat{a}}\coloneqq \bm a/\norm{\bm a}$. In formulas
\begin{align}
    &\min_{\bm a\in\mathcal{A}_{gs}(\bm r(h_f))}  (-\nabla U(\bm y))\cdot \Tilde g(\bm y,z_0-h_f,\bm a) \frac{\bm a}{\norm{\bm a}}= -\max_{\bm a\in\mathcal{A}_{gs}(\bm r(h_f))}  \nabla U(\bm y)\cdot \Tilde g(\bm y,z_0-h_f,\bm a) \frac{\bm a}{\norm{\bm a}}\\&=-\max_{\bm{\hat a}\in S^1}  \left[\nabla U(\bm y)\cdot\bm{\hat a}\max_{\substack{\norm{\bm a}\in\R_+ \\\bm a\in\mathcal{A}_{gs}(\bm r(h_f)) }}\Tilde g(\bm y,z_0-h_f,\bm{\hat a}\norm{\bm a})\right]\,.
\end{align}
In taking the maximum of $g$ with respect to the ground speed's magnitude we're supposing that the glider adapts its speed to maximize the glide ratio in direction $\bm{\hat a}$. 

Therefore defining $g(\bm y, z, \bm{\hat a})\coloneqq \max_{\substack{\norm{\bm a}\in\R_+ \\\bm a\in\mathcal{A}_{gs}((\bm y,z)) }}\Tilde g(\bm y,z,\bm{\hat a}\norm{\bm a})$ and replacing $z_0-h_f$ with $U(\bm y)$ since the two are equal, we can rewrite \eqref{eq:HJB_R2} in its definitive form together with the obstacle avoidance constraint and the initial condition:
\begin{align}
    \label{eq:HJB_S1}
    &\max_{\bm{\hat a}\in S^1} \nabla U(\bm y)\cdot \bm{\hat a} g\left(\bm y,z_0-U(\bm y),\bm{\hat a}\right) =1\\
    \label{eq:constraints_HJB_S1}
    &z_0-U(\bm y)\geq h_\text{min}(\bm y),\qquad U(\bm{x_0})=0
\end{align}
Equation \eqref{eq:HJB_S1} is a first order nonlinear PDE belonging to the well known class of static Hamilton-Jacobi PDEs \cite{sethian1996fast,sethian1999fast,sethian2003ordered}. In general, it does not admit a closed-form solution, however, very efficient algorithms exist to solve it approximately. Before moving on let us briefly comment on $g$. In our derivation we supposed that the aircraft flies at the speed that yields the optimal glide ratio, implying that the pilot should adapt the magnitude of the airspeed to the local wind. Alternatively, to obtain more conservative estimates one can assume instead that the aircraft flies at a constant horizontal airspeed $v$. In this case $g(\bm x,z,\bm{\hat a})=\norm{\Tilde{\bm a}+ \bm W(\bm x, z)}/s(v)$, where $\Tilde{\bm a}$ solves $\bm{\hat a}=\frac{\bm W(\bm x, z )+ \Tilde{\bm a}}{\norm{\bm W(\bm x, z )+ \Tilde{\bm a}}}$, with $\norm{\Tilde{\bm a}}=v$.
\subsubsection{Optimal paths}
Consider a point $\bm y\in\text{GRR}(\bm x_0,z_0)$ that we want to reach. How can we compute the optimal path going from $\bm x_0$ to $\bm y$ that results in the least loss of altitude?
The optimal control formulation allows us to easily answer this question. Consider the optimal path $\bm x:[0,h_f]\mapsto \R^2$ going from $\bm x(0)=\bm x_0$ to $\bm x(h_f)=\bm y$.
Then from \eqref{eq:bellman_opt_int_grrp} we see that the path's derivative in $\bm y$ is $\frac{d\bm x}{dh}(h_f)=\arg\min_{\bm a\in \mathcal{A}_{gs}(\bm r(h_f))}\left(-\nabla U(\bm y)\cdot\bm{\hat a}\Tilde g(\bm y,z_0-h_f,\bm a)\right)$. Since we're only interested in the direction $\bm{\hat a}$ and we don't care about the magnitude, we can define the vector field $\bm{\hat a}_G:\R^2\mapsto S^1$ 
\begin{equation}
\label{eq:opt_vect_field_grrp}
    \bm{\hat a}_G(\bm y)\coloneqq \arg\max_{\bm{\hat a}\in S^1}\nabla U(\bm y)\cdot \bm{\hat a} g\left(\bm y,z_0-U(\bm y),\bm{\hat a}\right)\,.
\end{equation}
$\bm{\hat a}_G(\bm y)$ is indeed the direction of the optimal path going from $\bm x_0$ to $\bm y$, in $\bm y$. 
To retrieve the whole trajectory it's sufficient to integrate the vector field backwards starting from $\bm y$.
Formally this means integrating the ODE
$\frac{d\bm \gamma}{du}(u)=-\bm{\hat a}_G (\bm \gamma(u))$ with initial condition $\bm \gamma (0)=\bm y$. 

\subsection{Gliding as front propagation}
 To better understand the meaning of equation \eqref{eq:HJB_S1} we shall consider the simpler case of zero wind and no obstacles. Thanks to the absence of wind, $\tilde g$ becomes a constant and the minimum can be taken explicitly: $\max_{\bm{\hat a}\in S^1} \nabla U(\bm y)\cdot \bm{\hat a}=\norm{\nabla U(\bm y)}$. Then \eqref{eq:HJB_S1} becomes
 \begin{equation}
 \label{eq:eikonal}
    \norm{\nabla U(\bm y)}=\frac{1}{g}.
 \end{equation}
which is known as the eikonal equation. The initial condition is $U(\bm x_0)=z_0$. 
The eikonal equation arises when describing the arrival times of a front at each point in space. Suppose we have a sheet of paper and at time $t=0$ we set fire to it at a point $\bm x_0\in \R^2$. Suppose further that the fire front advances through the paper with speed $g$. Then letting $U(\bm y)$ be the time at which the fire front reaches $\bm y$, we have that $U$ satisfies \eqref{eq:eikonal}. We can also include obstacles in this analogy. 
Suppose some portion of the paper is wet. In this case, the fire front stops at the border of the wet region. This is similar to what happens when the glider's altitude falls below $h_\text{min}$. 
This analogy shows how we can see the descending glider as an expanding front that passes around the obstacles until it collides with them and stops.
\subsection{Minimal return altitude problem}
\label{sec:mrap_main}
In MRAP we want to compute the minimal altitude over a position $\bm y\in \R^2$ to safely return to an airfield at position $\bm x_a\in\R^2$.
This is equivalent to computing the minimum altitude $z_0$, such that $\bm x_a\in \text{GRR}(\bm y,z_0)$. We encode this in the minimal return altitude function $V:\R^2\mapsto \R$ defined as 
\begin{equation}
    \label{eq:V_mrap}
    V(\bm y)=\min_{z_0} \left\{ z_0 \text{ s.t. } \bm x_a\in \text{GRR}(\bm y,z_0)\right\}.
\end{equation} 
In Appendix \ref{app:mrap} we present the full derivation of the HJB equations for MRAP, we report below the final equations for $V$:

  \begin{align}
    \label{eq:HJB_mrap}
      &\begin{cases}
    \max_{\bm{\hat a}\in S^1}\nabla V(\bm y)\cdot \bm{\hat a}\,\left[\max\left(1/g(\bm y,V(\bm y),-\bm{\hat a}), \;\nabla h_\text{min}(\bm y)\cdot \bm{\hat a}\right)\right]^{-1}=1 & \text{if } V(\bm y)=h_\text{min}(\bm y) \\
     \max_{\bm{\hat a}\in S^1}\nabla V(\bm y)\cdot \bm{\hat a}\,g(\bm y, V(\bm y),-\bm{\hat a})=1, & \text{if } V(\bm y)>h_\text{min}(\bm y)\,,
  \end{cases}\\
  \label{eq:HJB_mrap_inicond}
  &V(\bm x_a)=h_\text{min}(\bm x_a).
  \end{align}
To give some intuition about this equation let us suppose the wind is zero, and therefore $g$ is constant. Suppose also that there are no obstacles ($h_\text{min}=0$).
In this case the solution is $V(\bm y)=\frac{1}{g}\norm{\bm y-\bm x_a}$, that is, a cone of slope $1/g$ with the vertex at $\bm x_a$. When obstacles are present, the cone solution can fall below $h_\text{min}$, in this case, it's not valid anymore and $V$ instead starts tracking the profile of $h_\text{min}$. See figure \ref{fig:mrap_1d_sketch} for a sketch of $V$ in the one-dimensional case. 
One can see the process of solving MRAP as a front propagating outwards from $\bm x_a$, where the arrival time of the front is the minimal return altitude. Similarly to GRRP, in MRAP one defines the vector field $\bm{\hat a}_M:\R^2\mapsto S^1$, where $\bm{\hat a}_M(\bm y)$ is the vector that attains the maximum in \eqref{eq:HJB_mrap}. Integrating $-\bm{\hat a}_M(\bm y)$ from the aircraft position, gives a feasible path terminating in $\bm x_a$. Feasible means that in all points the path does not go below $h_\text{min}$ and that the steepness is at least $1/g$. This path is however different in general from the path that minimizes the altitude loss to reach $\bm x_a$. In appendix \ref{app:optimal_traj_mrap_grrp} we compare the trajectories found in GRRP with those of MRAP and comment on their differences.
\section{Algorithms}
Equations \eqref{eq:HJB_S1} and \eqref{eq:HJB_mrap} can be solved analytically only in the case of uniform wind and no obstacles. We present their solution in Appendix \ref{app:analytic_grrp_mrap}. 
In all other cases, a numerical approach is needed. Algorithms to solve the HJB equations of the form \eqref{eq:HJB_S1} in the absence of constraints have been developed by Sethian and Vladimirsky in \cite{sethian2003ordered,sethian1996fast,sethian1999fast}. 
These methods are based on discretizing the 2d plane using a grid of points and subsequently approximating the value of $U$ at each grid point. The algorithms mimic the front propagation by setting the values of $U$ starting from the nodes closer to the initial point (respectively the aircraft position and the airfield position in MRAP and GRRP) and expanding outwards.

Glikonal-G, the algorithm to solve GRRP, takes as input the initial position and altitude of the glider,  the elevation profile, the wind vector field, and the function $g(\bm y,z, \bm{\hat a})$ giving the glide ratio at $(\bm y,z)$ in direction $\bm{\hat a}$. It then computes an approximation of $U$ (solving equations  \eqref{eq:HJB_S1},\eqref{eq:constraints_HJB_S1}) on a grid. Appendix \ref{app:algos_GRRP} contains the pseudocode of a version of Glikonal-G for the special case of no wind. This pseudocode is based on the FMM scheme and therefore does not reflect the true implementation of Glikonal-G which is instead OUM based. We include it to familiarize the reader with the logic of a simpler version of the algorithm.

Glikonal-M, the algorithm to solve MRAP, is only able to deal with the zero wind case. This is because solving the equations with an arbitrary wind field presents some additional technical challenges. In Appendix \ref{sec:simplifications_hjb_mrap} we derive the simplified form of the HJB equations \eqref{eq:HJB_mrap} for the windless case.
Glikonal-M takes as input the position of the airfield, the elevation profile, and the glide ratio $g\in\R$ (which is constant in the absence of wind). It outputs an approximation of $V$ solving \eqref{eq:HJB_mrap},\eqref{eq:HJB_mrap_inicond} on a grid. The pseudocode of Glikonal-M is reported in Appendix \ref{app:algo_mrap}. 

The grid spacing is a parameter of both algorithms. The finer the grid, the smaller the error, the longer it takes to run. When the grid spacing goes to zero both algorithms exactly solve the respective PDEs. For a proof of this fact see \cite{sethian2003ordered,sethian1996fast}. The time complexity of both algorithms is $O(n\log n)$, where $n$ is the number of points in the grid. 

\section{Numerical results}
In this section, we present the numerical results obtained with Glikonal-G and Glikonal-M in several settings. The code to reproduce all experiments can be found at \url{https://github.com/giovannipiccioli/glikonal}.
\subsection{Glikonal-G}
\begin{figure}[h]
    \centering
    \includegraphics[width=\textwidth]{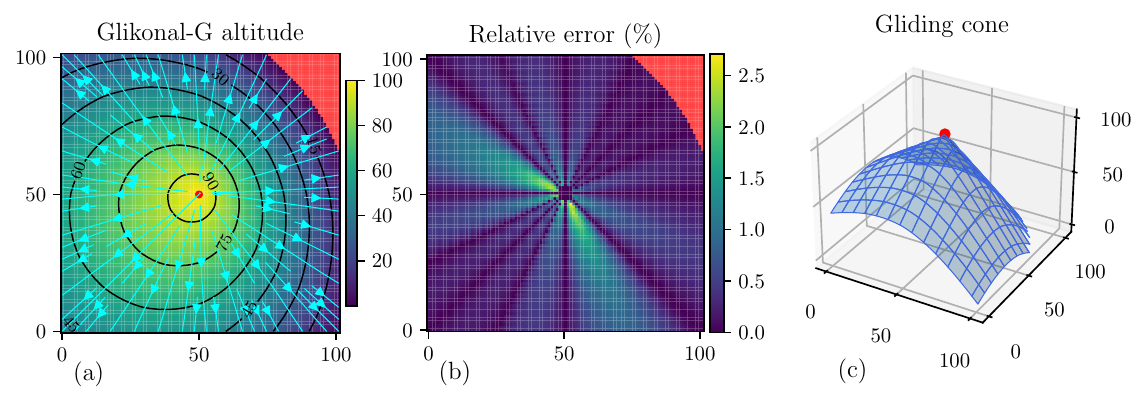}
    \caption{Result of running Glikonal-G on flat terrain with uniform wind. The red dot represents the initial position of the glider. \textbf{(a)} Heatmap and contour lines of the function $z_0-U_G$. The red region is outside the gliding reachable region. The turquoise lines are the optimal trajectories. \textbf{(b)} Relative error of the Glikonal-G solution, i.e., $(U_G-U)/U$. \textbf{(c)} 3d plot of the function $z_0-U_G$.}
    \label{fig:flat_glikonal_g}
\end{figure}

In the first set of numerical experiments, we assume for simplicity (and for visualization purposes) that the glide ratio is $1$ in the absence of wind and that the aircraft's airspeed is also $1$. When the wind is present we assume that the aircraft travels in every direction at fixed airspeed, therefore $g(\bm y,z,\bm{\hat a})=\norm{\bm{\hat a} + \bm W(\bm y,z)}$. Notice that this is not the optimal glide ratio, as in principle, to maximize the glide ratio, one should increase the airspeed when moving upwind. \footnote{We choose not to optimize the glide ratio with respect to the airspeed since this would force us to specify the functional form of the glide polar $s:\R\mapsto (-\infty,0)$ of the aircraft. Moreover, it is reasonable to assume that a pilot would try to maintain a constant airspeed.} In the following we will refer to $U$ as the solution of GRRP, and to $U_G$ as its approximation returned by Glikonal-G. We run our experiments on a $101\times 101$ grid with horizontal and vertical spacings of $1$. In all experiments, the initial altitude is taken to be $z_0=100$. 

We start from a flat elevation profile, that is $h_\text{min}(\bm y)=0$ everywhere, and a uniform (i.e. constant in space) wind vector field with norm  $\norm{\bm W}=0.6$ and direction 240°. In this setting the analytic solution $U$ can be computed (see Appendix \ref{app:analytic_grrp_mrap}).
The exact solution is computed within a distance of 2.9 of the initial point to initialize the algorithm.
Figure \ref{fig:flat_glikonal_g} illustrates the results.

Panel (a) depicts the values of $z_0-U_G$, and its contour lines. The red dot is the initial position of the glider. As expected, the altitude decreases faster when travelling upwind, than downwind. The turquoise lines represent the optimal trajectories, and the red zone is the region that cannot be reached in gliding flight. In the absence of obstacles and with constant wind, the optimal trajectories are straight as expected under the analytic solution. Panel (b) shows the relative error $(U_G-U)/U$ of the algorithmic approximation. We observe that $U_G$'s relative error is always below 3\%. Moreover the error is always positive, implying that $U_G>U$, and therefore the algorithm errs on the conservative side, overestimating the lost altitude. In Appendix \ref{app:error_algo} the total and relative errors of the algorithm are computed in a number of settings. In all the cases we consider, the error is below 4\%, and it is always on the conservative side.
Finally in panel (c) the 3d  'gliding cone', i.e., the function $z_0-U_G$ is plotted. One can again observe how the cone is steeper in the upwind direction.

\begin{figure}[h]
    \centering
    \includegraphics[width=\textwidth]{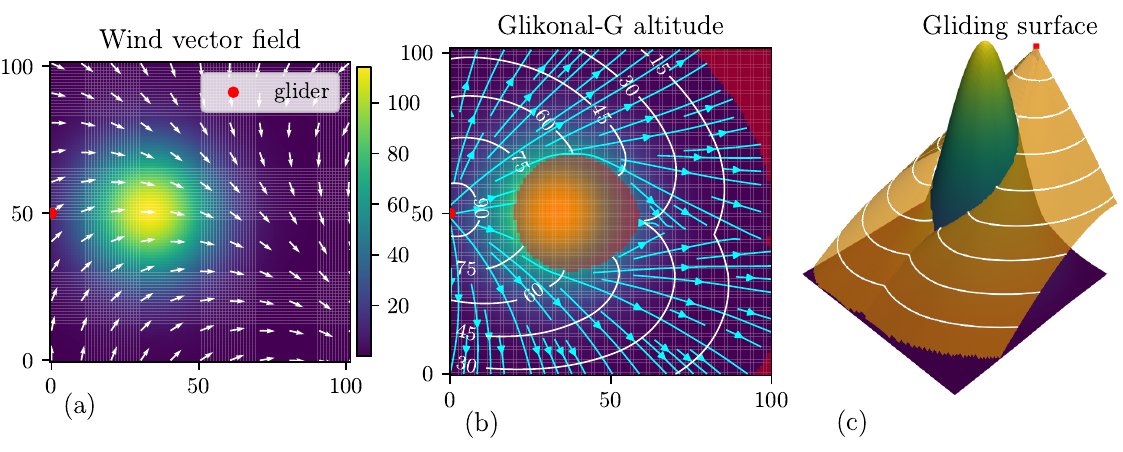}
    \caption{Glikonal-G's solution in the case of non-uniform wind and single mountain peak. The red dot represents the glider's initial position. \textbf{(a)} Wind direction (constant with altitude) plotted on top of the elevation profile (displayed as a heatmap). \textbf{(b)} Contour lines of the function $z_0-U_G$ (in white) and optimal trajectories (in turquoise). The heatmap is the elevation profile, and the red-shaded regions are those not reachable in gliding flight. \textbf{(c)} 3d representation of the function $z_0-U_G$, in orange and of $h_\text{min}$.}
    \label{fig:single_peak}
\end{figure}
We now turn to a more complex example. We consider an elevation profile $h_\text{min}$ with a single peak, whose heatmap is displayed in figure \ref{fig:single_peak} (a), together with the wind direction throughout space. We take the wind to be nonuniform both in intensity and direction. The wind intensity, $\norm{\bm W}$, instead starts at 0.54 at an altitude of 100 and linearly decreases to 0.48 at zero altitude.

The output of the Glikonal-G algorithm is displayed in panel (b). First, we remark that the optimal trajectories are forced to go around the obstacle, causing them to bend. The bending of the optimal trajectories is also caused by the varying wind: for example looking at the bottom left corner we clearly see that optimal paths turn clockwise. Panel (c) presents a 3d depiction of the function $z_0-U_G$, together with the elevation profile.
In Appendix \ref{app:additional_numerics} we present the results for another numerical experiment with artificial data, involving a mountain range with two saddles.

\subsubsection{GRRP on real data}
To conclude our study of the Glikonal-G algorithm, we present a real-world example, with elevation data taken from \cite{JAXA2021}. We consider a glider with a glide ratio of 20 in the absence of wind, moving at a fixed airspeed of 100km/h. For this experiment, we use a $295\times 305$ grid with a spacing of $150$ m.
The wind, represented with white arrows in figure \ref{fig:grrp_lstr} (a), is uniform and forms an angle of 30° with the south direction; its strength is $15$ km/h below $750$ m, then increases linearly to $40$ km/h at $2000$ m, and stays constant at higher altitudes. In the same panel, the elevation profile is shown. We require the glider to be at more than 200m above the ground in every point of the trajectory. The initial position of the glider is (46.59540N, 6.40050E) at an altitude of $z_0=2070$ m (red dot). The purple dot represents the airfield of Montricher, the closest one to the glider. The airfield is located on the opposite side of the Jura mountain range. 
Panel (b) depicts the contour lines of the function $z_0-U_G$ in black, and the optimal trajectories in blue. To reach the airfield, one must first pass the mountain range to the south and then turn sharply left. Since the straight path from the glider to the airfield is obstructed, we remark that any line-of-sight algorithm would fail to recognize that the airfield is within gliding range. Panel (c) shows a three dimensional visualization of $z_0-U_G$ and of the elevation profile.
\subsubsection{Computational efficiency}
Using a C++ implementation of Glikonal-G we characterized the running time of the algorithm. 
To compute both the optimal trajectories and the function $U_G$ the algorithm takes on average time of $0.16$ s and $0.14$ s respectively for figures \ref{fig:flat_glikonal_g} and \ref{fig:single_peak}. The experiments were run on a single core of an Intel Core i7-9750H on a MacBook Pro laptop. We can confidently claim that Glikonal-G can be run in real time on an aircraft's avionics.
\begin{figure}[h]
    \centering
    \includegraphics[width=0.8\textwidth]{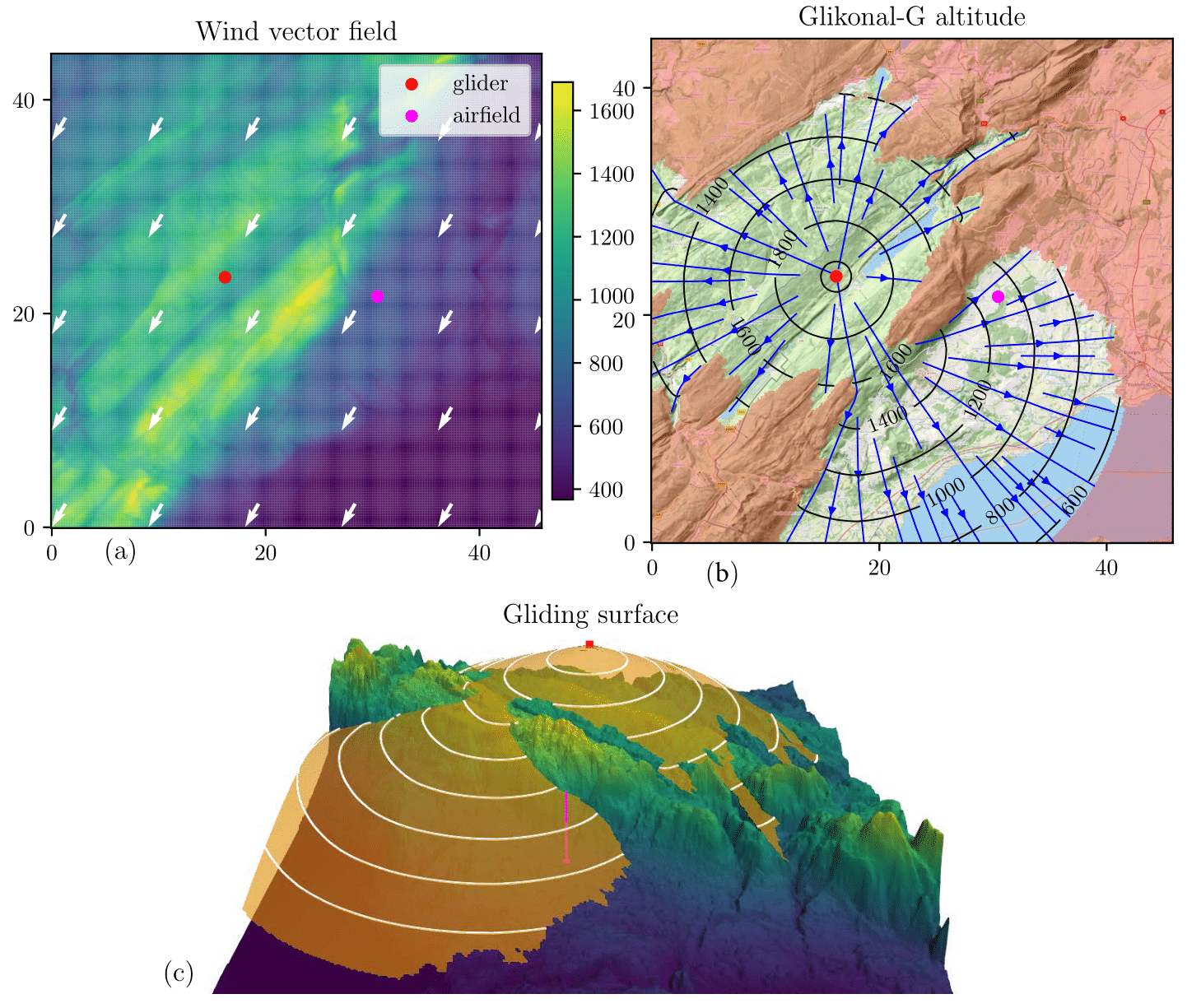}
    \caption{Glikonal-G on a real-world example. Horizontal distance is measured in kilometers, while the altitude is in meters (see colorbar). The red and purple dots represent respectively the glider's and airfield's positions. \textbf{(a)} Heatmap of the elevation profile of the Jura mountain range. The glider and the closest airfield are located on opposite sides of the range. The white arrows represent the uniform wind across the map. \textbf{(b)} Contour lines of the altitude function $z_0-U_G$ in black and optimal trajectories in blue. The red-shaded region is unreachable. Background map from \cite{OpenStreetMap}. \textbf{(c)} 3d representation of the altitude of the glider ($z_0-U_G$) in orange, on top of the elevation profile. for visualization purposes, the vertical axis is rescaled by 8 in comparison to the horizontal axes. The purple vertical line highlights the position of the airfield.}
    \label{fig:grrp_lstr}
\end{figure}
\subsection{Glikonal-M}
To test the Glikonal-M algorithm we use it to approximate the function $V$, defined in \eqref{eq:V_mrap}. We name $V_G$ the approximation of $V$ returned by the algorithm. We run our experiments on a $101\times 101$ grid with horizontal and vertical spacings of $1$. We start with the case of a flat terrain ($h_\text{min}=0$) and glide ratio equal to one. The top row in figure \ref{fig:mrap_flat_single_peak} depicts the results of this experiment. In (a) the contour lines and heatmap of $V_G$ are plotted, together with the feasible re-entry trajectories, which in this case extend radially from the airfield. Using the results in Appendix \ref{app:analytic_grrp_mrap} we compute the analytic solution $V$ and use it to plot the relative error $(V_G-V)/V$ in panel (b). The error is below 4\% in the whole space, moreover, $V_G$ is always greater than $V$, thus erring on the safe side. In Appendix \ref{app:error_algo} the error is computed for several different settings in which the analytical solution is available. In all the cases the error is below 5\% and always on the conservative side.
Panel (c) concludes this example by showing the function $V_G$, which takes the shape of a cone having the tip at the airfield. A glider above the cone will be able to reach the airfield, one below will not.

In the next setting, whose results are depicted in the bottom row of figure \ref{fig:mrap_flat_single_peak}, we consider the case of an elevation profile with a single peak. The glide ratio is again set to 1. In (d) the elevation is shown as a heatmap and the contour lines of $V_G$ are plotted on top. The turquoise lines are the feasible re-entry trajectories. Notice how the trajectories bend around the mountain to avoid it. Panel (e) presents a 3d plot of the elevation profile and of the function $V_G$ (in orange). Appendix \ref{app:additional_numerics} contains the results of another numerical experiment with artificial data, where the elevation profile includes a mountain range with two saddles. 
\begin{figure}[h]
    \centering
    \includegraphics[width=0.9\linewidth]{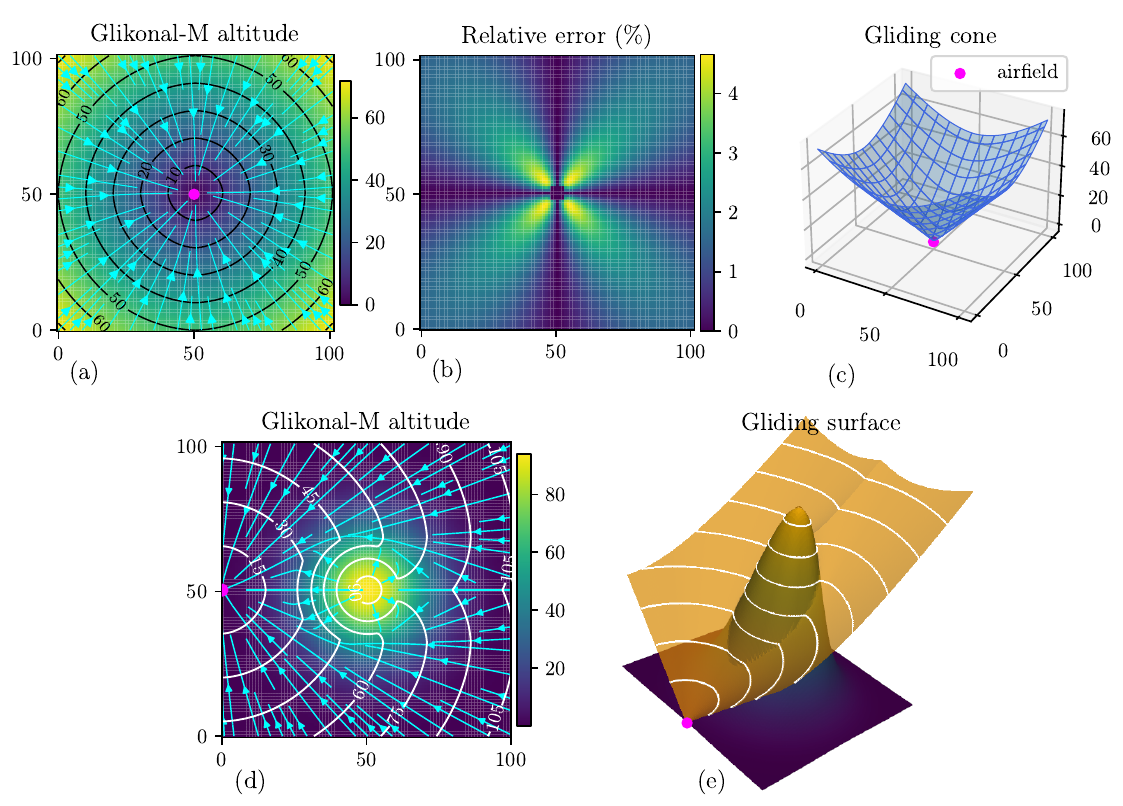}
    \caption{Numerical experiments on Glikonal-M. The purple dot indicates the position of the airfield. \textbf{Top row:} results on flat terrain. \textbf{(a)} Heatmap and contour lines of the function $V_G$. The turquoise lines are feasible re-entry trajectories to the airfield. \textbf{(b)} Relative error of the Glikonal-M solution, i.e. $(V_G-V)/V$. \textbf{(c)} 3d plot of the function $V_G$.
    \textbf{Bottom row:} elevation profile with a single mountain peak. \textbf{(d)} Heatmap of the elevation profile with contour lines of the function $V_G$ (in white) and re-entry trajectories (in turquoise). \textbf{(e)} 3d plot of the function $V_G$ in orange on top of the elevation profile.}
    \label{fig:mrap_flat_single_peak}
\end{figure}
\subsubsection{MRAP on real data}
In this experiment, we apply the Glikonal-M algorithm to study the minimal altitude needed to reach the airfield of Montricher in Switzerland, represented by a purple dot in figure \ref{fig:mrap_lstr}. The elevation profile of this region is given as a heatmap in figure \ref{fig:grrp_lstr} (a). The size of the grid and the spacing are the same as in the Glikonal-G case. $h_\text{min}$ is set to the terrain elevation plus $200$ m, and the glide ratio is 20.
In (a) the contour lines are plotted in black, and the re-entry trajectories in blue. There are no obstacles to the south and west of the airfield, hence the contour lines are circle arcs and the re-entry paths are just straight lines. To the northwest, the minimal return altitude is increased due to the presence of mountains, which one must pass to reach the airfield. Comparing the feasible paths in (a) with the elevation profile in \ref{fig:grrp_lstr} (a) we remark that the trajectories found by solving MRAP avoid passing on the points with highest elevation, and instead route the glider through the mountain passes.
\begin{figure}[h]
    \centering
    \includegraphics[width=0.9\linewidth]{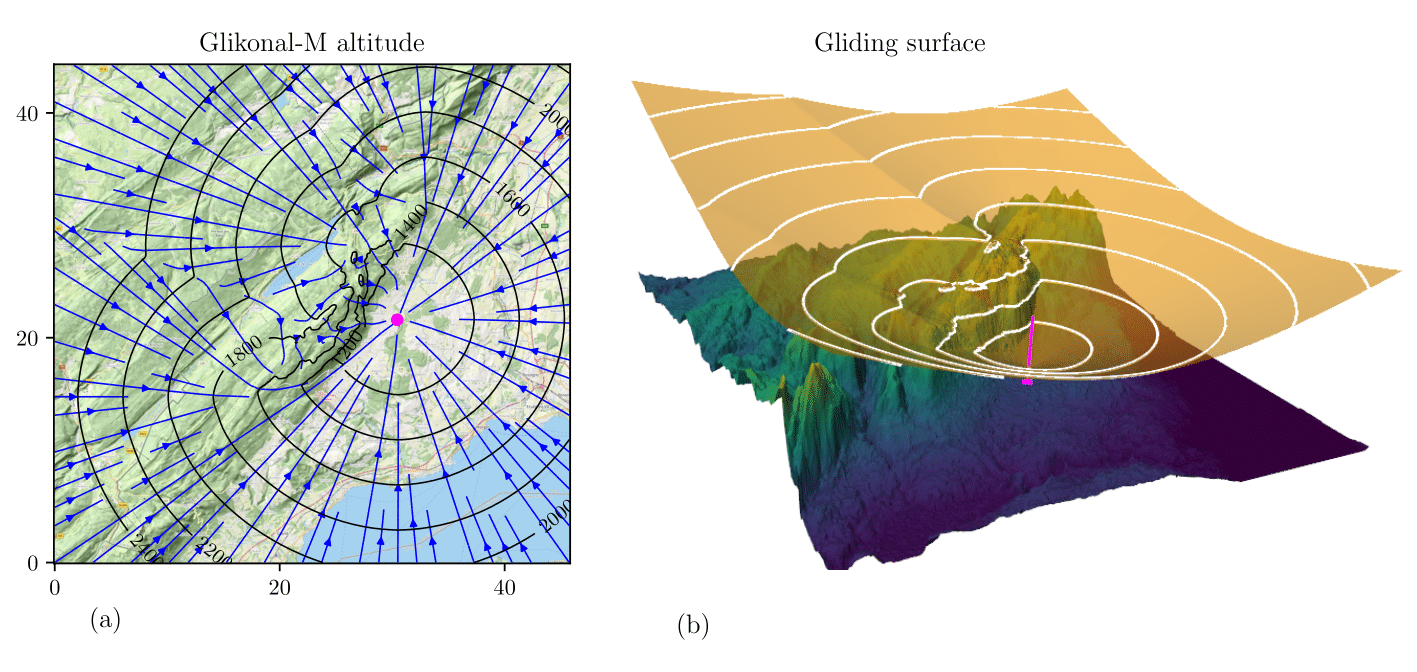}
    \caption{Glikonal-M on a real-world example. The horizontal distances are measured in kilometers. The purple dot represents the airfield's position. \textbf{(a)} Contour lines of the minimal altitude function $V_G$ in black and feasible trajectories in blue. \textbf{(b)} 3d representation of the function $V_G$ in orange, on top of the elevation profile. For visualization purposes, the vertical axis is rescaled by 8 in comparison to the horizontal axes. The purple vertical line highlights the position of the airfield.}
    \label{fig:mrap_lstr}
\end{figure}
\section{Conclusions}
In this work, we introduced the GRRP and MRAP problems and presented the Glikonal-G and Glikonal-M algorithms to solve them. These respectively compute the GRR from a given position and altitude, and the minimal altitude to glide to a given airfield. First, the aircraft dynamics is modeled by assuming that discontinuous changes in ground speed are allowed. Then an optimal control formulation of the aforementioned problems is presented. The HJB PDEs are derived, and then solved in a discretized form using Glikonal-G and Glikonal-M. The precision of the solution is verified in several settings, and results are presented both in the case of artificial and real-world data. 

Glikonal-G is fast enough to run in real time on an on-board computer and display the reachable region to the pilot throughout the flight. Moreover, it is more advanced than other algorithms on several fronts. First, it surpasses line-of-sight methods, which only consider straight gliding paths (at most with an initial turn to the desired heading). In several situations (see figures \ref{fig:grrp_lstr}, \ref{fig:grrp_mountain_range}, \ref{fig:single_peak}) the GRR found with these methods is incomplete. The optimal trajectories in Glikonal-G can instead turn around obstacles, therefore outputting the correct GRR. Second, it can handle any wind vector field, by having the glide ratio depend arbitrarily on the position, altitude, and course of the glider. 

Glikonal-M is able to solve MRAP for arbitrary elevation profiles. Knowing in advance the altitude needed to glide to an airfield can help pilots plan their flights and always have an available landing option in case of engine failure. An implementation of Glikonal-M which can deal with wind is the object of future research.

The main limitation of our aircraft dynamics model is its inability to account for the nonzero turning radius and the increased sink rate during turns at a fixed airspeed. We hope to overcome these impediments in future work, starting from accounting for the altitude loss in the initial turn.

\begin{acknowledgments}
I am grateful to the people of Montricher's airfield for inspiring this work. I would also like to thank Lenka Zdeborová for granting me the freedom to work on this project.
\end{acknowledgments}

\bibliography{apssamp}

\appendix
\section{Optimal control formulation of MRAP}
\label{app:mrap}
 We will first define MRAP in terms of the minimal altitude for which the airport falls within the GRR. Then we show how MRAP can be formulated as an optimal control problem for an ascending dynamics that from the minimal altitude over the airport climbs to the position of the aircraft. This corresponds to a time reversal of the descending dynamics.
\subsection{MRAP from GRRP}
In MRAP we want to compute the minimal altitude over a position $\bm y\in \R^2$ to safely return to an airport at position $\bm x_a\in\R^2$. 
We encode this in the minimal return function $V:\R^2\mapsto \R$ defined as 
\begin{equation}
    V(\bm y)=\min_{z_0} \left\{ z_0 \text{ s.t. } \bm x_a\in \text{GRR}(\bm y,z_0)\right\},
\end{equation}
Where GRR($\bm y,z_0$) is the set of positions safely reachable in gliding flight when starting at altitude $z_0$ over $\bm y$, as defined in \eqref{eq:GRR}.
More explicitly this can be written as 
\begin{align}
\label{eq:mrap_app}
    V(\bm y)&=\min_{z_0, \bm a\in \mathcal F} z_0 \,\\
    \label{eq:constraints_mrap}
    \text{subject to: eq. \eqref{eq:dyn_h}}, \,\bm x(0)=\bm y, \,\bm x(h_f)=\bm x_a, \,&z_0-h\geq h_\text{min}(\bm x(h)) \,\forall h\in [0,h_f],\,\bm a(h)\in \mathcal A_{gs}(\bm r(h)) \,\forall h\in[0,h_f],
\end{align}
where $h_f$ here is a free parameter and represents the altitude loss to reach the airport.
MRAP basically consists of computing $V$ in the whole space. Intuitively, the further we are from $\bm x_a$, the larger $V$. However, this formulation cannot be readily translated into an algorithm that computes the $V$, since we're constraining the final point of the dynamics instead of the initial one. For this reason, we now present a sequence of equivalent problems with the aim of translating \eqref{eq:mrap_app} into a problem with a convenient algorithmic solution.

Consider the relaxed aircraft dynamics 
\begin{equation}
\label{eq:dyn_h_relaxed}
    \frac{d\bm x}{dh} (h)\in\left\{\lambda \Tilde g(\bm x(h),z_0-h,\bm a(h))\bm{\hat a}(h) \text{ s.t. }\lambda\in[0,1]\right\}.
\end{equation}
with the usual definition $\bm{\hat a}=\bm a/\norm{\bm a}$ and $h$ being the lost altitude. 
This dynamics is identical to \eqref{eq:dyn_h} except for the addition of the factor $\lambda$ which allows the glide ratio to be smaller than $\Tilde g$. This means that the aircraft can descend with glide ratio at most $\Tilde g$, but can also descend more steeply. 

Define $V_2:\R^2\mapsto \R$ is the same exact way as the function $V$ (i.e. using \eqref{eq:mrap_app}) but with the relaxed dynamics \eqref{eq:dyn_h_relaxed} in place of \eqref{eq:dyn_h} in the constraints \eqref{eq:constraints_mrap}.
In formulas
\begin{align}
\label{eq:mrap_app_V2}
    V_2(\bm y)&=\min_{z_0, \bm a\in \mathcal F,\bm x\in \mathcal F} z_0 \,\\
    \label{eq:constraints_mrap_V2}
    \text{subject to: eq. \eqref{eq:dyn_h_relaxed}}, \,\bm x(0)=\bm y, \,\bm x(h_f)=\bm x_a, \,&z_0-h\geq h_\text{min}(\bm x(h)) \,\forall h\in [0,h_f],\,\bm a(h)\in \mathcal A_{gs}(\bm r(h)) \,\forall h\in[0,h_f],
\end{align}
Notice that in this case specifying $\bm a$ does not uniquely determine a curve $\bm x$, for this reason in \eqref{eq:mrap_app_V2} we take the minimum also over possible curves  $\bm x:\R_+\mapsto \R^2$.

We now prove that $V_2(\bm y)=V(\bm y)$ for all $\bm y$.

\begin{lemma}
    For all $\bm y\in\R^2$, $V_2(\bm y)=V(\bm y)$. 
\end{lemma}
\begin{proof}
The set of possible paths allowed by the relaxed dynamics \eqref{eq:dyn_h_relaxed} contains the set of paths allowed under the \eqref{eq:dyn_h}. Therefore $V_2\leq V$.
We now prove that $V\leq V_2$.
The idea is very simple: given a path that is steeper than $1/g$ is some section, we can always replace it with a path that has the same horizontal projection, (i.e. passes through the same points in the horizontal plane) but maintains a higher altitude than the original path by having steepness exactly $1/g$. If the original path satisfies the constraints, so does the modified path, since it passes higher than the first. Then taking the minimum over paths with the dynamics \eqref{eq:dyn_h_relaxed} is equivalent to taking the minimum over paths with the dynamics \eqref{eq:dyn_h}.

We now formalize the above.
Consider a path $\bm x: \R_+\mapsto \R^2$ satisfying \eqref{eq:constraints_mrap_V2}.
We now reparametrize $\bm x$ in order to satisfy the non-relaxed dynamics.
Define $\Bar h$ as the minimal altitude the aircraft can lose when following the path $\bm x$. $\Bar h$ obeys the equation
\begin{equation}
    \Bar{h}(h)=\int_0^h dh'\frac{\norm{d\bm x/dh(h')}}{\Tilde g(\bm x(h'),z_0-\Bar{h}(h'),\bm a(h'))}
\end{equation}
Notice this cannot be used to directly to compute $\Bar h$, since it appears on both sides of the expression. We introduce the reparametrized path $\Bar{\bm x}(h')\coloneqq\bm x(\Bar h^{-1}(h'))$. By computing its derivative we get
\begin{equation}
    \frac{d\Bar{\bm x}}{dh'}(h')=\frac{d\bm x}{dh}(\Bar h^{-1}(h'))\frac{d\Bar h^{-1}}{dh'}(h')=\bm{\hat a} \,\Tilde g(\bm x(h'),z_0-h',\bm a(h')),
\end{equation}
and thus $\Bar{\bm x}$ satisfies the dynamics \eqref{eq:dyn_h}.
At this point, we remark that $\Bar{h}(h)\leq h$. To see this suppose that for a certain $h$, $\Bar{h}(h)=h$ (this happens for example when $h=\Bar{h}=0$). Then deriving $\Bar{h}$ we get
\begin{equation}
    \frac{d\Bar{h}}{dh}(h)=\frac{\norm{d\bm x/dh(h)}}{\Tilde g(\bm x(h),z_0-\Bar{h}(h))}=\frac{\norm{d\bm x/dh(h)}}{\Tilde g(\bm x(h),z_0-h)}\leq 1,
\end{equation}
where in the second passage we used that $\Bar{h}(h)=h$, and the inequality holds because $\norm{d\bm x/dh}\leq g(\bm x(h),z_0-h,\bm a(h))$, according to \eqref{eq:dyn_h_relaxed}.

Since $\Bar{h}(h)\leq h$, $\Bar{\bm x}$ satisfies the terrain avoidance constraints if $\bm x$ does, since $\Bar{\bm x}$ passes at least as high as $\bm x$.

The consequence of this is that if there exists an $\bm x$ satisfying \eqref{eq:constraints_mrap_V2}, then there exists a reparametrization $\Bar{\bm x}$ of $\bm x$ that still satisfies \eqref{eq:constraints_mrap}.
Therefore $V_2 \leq V$.
\end{proof}

Now that the equality between $V$ and $V_2$ has been established, we shall introduce the time-reversed problem. This dynamics starts at an altitude $\Bar z_0$ and climbs with steepness at least $g$. $h$ here represents the \textit{gained} altitude from the initial point.

\begin{equation}
\label{eq:dyn_h_relaxed_reversed}
    \frac{d\bm x}{dh} (h)\in\left\{\lambda \Tilde g(\bm x(h),\Bar z_0+h,-\bm a(h))\bm{\hat a}(h) \text{ s.t. }\lambda\in[0,1]\right\}.
\end{equation}
We now introduce the function $V_3:\R^2\mapsto \R$ which represents the minimal altitude at which the climbing aircraft can reach a point. The optimal control problem for $V_3$ is
\begin{align}
\label{eq:mrap_app_V3}
    V_3(\bm y)&=\min_{\Bar z_0, \bm a\in \mathcal F,\bm x\in \mathcal F} \Bar z_0+h_f \,\\
    \label{eq:constraints_mrap_V3}
    \text{subject to: eq. \eqref{eq:dyn_h_relaxed_reversed}}, \,\bm x(0)=\bm x_a, \,\bm x(h_f)=\bm y, \,&\Bar z_0+ h\geq h_\text{min}(\bm x(h)) \,\forall h\in [0,h_f],\,\bm a(h)\in -\mathcal A_{gs}(\bm r(h)) \,\forall h\in[0,h_f],
\end{align}
where $-\mathcal A_{gs}(\bm r(h))=\left\{-\bm a, \; \bm a\in \mathcal A_{gs}(\bm r(h))\right\}$.
In this problem the aircraft starts at an altitude $z_0$ over the airport and then after climbing $h_f$ meters, it reaches $\bm y$. Let us say it once more: $h$ here is the \textit{gained} altitude. The dynamics \eqref{eq:dyn_h_relaxed} imposes that the aircraft must climb with steepness at least $g$.
We now prove that $V_3(\bm y)=V_2(\bm y)$, for all $\bm y$.
\begin{lemma}
     For all $\bm y\in\R^2$, $V_3(\bm y)=V_2(\bm y)$. 
\end{lemma}
\begin{proof}
    First notice that $\Bar{z}_0+h_f=z_0$, thus we're minimizing the same objective function. We're left to prove that the constraints \eqref{eq:constraints_mrap_V2} and \eqref{eq:constraints_mrap_V3} define the same set of paths. To do this we will prove that a path satisfies \eqref{eq:constraints_mrap_V2} if and only if its time reversal satisfies \eqref{eq:constraints_mrap_V3}.
    
    Let $\bm x_2(h)$ be a path satisfying \eqref{eq:constraints_mrap_V2}.
    Consider the time reversed path $\bm x_3(h)\coloneqq \bm x_2(h_f-h)$.
    \begin{itemize}
        \item $\bm x_2(0)=\bm y\; \iff \;\bm x_3(h_f)=\bm y$
        \item $\bm x_2(h_f)=\bm x_a\; \iff \;\bm x_3(0)=\bm x_a$
        \item We now check that $\bm x_2$ satisfies the terrain avoidance constraint if and only if $\bm x_2$ does. Recall that $z_0=\Bar z_0+h_f$.
        \begin{align}
            &z_0-h\geq h_\text{min}(\bm x_2(h)),\, \forall h\in[0,h_f]\iff \Bar z_0+h_f-h\geq  h_\text{min}(\bm x_2(h)),\, \forall h\in[0,h_f]\iff \\&\Bar z_0+h_f-h\geq  h_\text{min}(\bm x_3(h_f-h)),\, \forall h\in[0,h_f]\iff \Bar z_0+h'\geq  h_\text{min}(\bm x_3(h')),\, \forall h'\in[0,h_f],
        \end{align}
        where in the last passage we renamed $h'=h_f-h$.
        \item  We shall now verify that  $\bm x_2$ obeys \eqref{eq:dyn_h_relaxed} iff $\bm x_3$ obeys \eqref{eq:dyn_h_relaxed_reversed}.
        \begin{align}
            &\frac{d\bm x_3}{dh} (h)=-\frac{d\bm x_2}{dh} (h_f-h)\in\left\{-\lambda \Tilde g(\bm x_2(h_f-h),z_0-h_f+h,\bm a(h_f-h))\bm{\hat a}(h) \text{ s.t. }\lambda\in[0,1]\right\}=\\&\left\{-\lambda \Tilde g(\bm x_3(h),\Bar z_0+h,\bm a(h))\bm{\hat a}(h) \text{ s.t. }\lambda\in[0,1]\right\}=\left\{\lambda \Tilde g(\bm x_3(h),\Bar z_0+h,-\bm a'(h))\bm{\hat a}'(h)\text{ s.t. }\lambda\in[0,1]\right\},
        \end{align}
        where $\bm a'=-\bm a$.
        In the second passage we used \eqref{eq:dyn_h_relaxed}, and in the third we used that $z_0=\Bar z_0+h_f$. In the final passage, we obtained \eqref{eq:dyn_h_relaxed_reversed} by taking into account that the vector $\bm a$ should be reversed since we're travelling in the opposite direction.
    \end{itemize}
\end{proof}
To summarize, in this section we proved that the more intuitive formulation of MRAP given by \eqref{eq:mrap_app} is equivalent to the problem \eqref{eq:mrap_app_V3}. 
We've proved that all the constraints are equivalent. Together with the fact that the objective function is the same this proves that the two optimization problems have the same optimal value.
\subsection{Derivaton of HJB equations for MRAP}
To better understand this derivation we advise the reader to first go through section \ref{sec:GRRP_theory}.
In this section, we rename $V_3\to V$ in order to simplify the notation.
Let $\bm x:[0,h_f]\mapsto \R^2$ be the path joining $\bm x(0)=\bm x_a$ and $\bm x(h_f)=\bm y$ that achieves the minimum in the problem \eqref{eq:mrap_app_V3}.
Then we can apply Bellman's optimality principle. In the following we name $\bm \nu \coloneqq \frac{d\bm x}{dh}(h_f)$. Notice $\bm \nu\in \R^2$.
\begin{align}
    V(\bm y)&= \min_{\Bar z_0, \bm x\in \mathcal F,\bm a\in \mathcal F} \Bar z_0+\int_0^{h_f} dh=V(\bm x(h_f-\epsilon))+\min_{\bm x\in \mathcal F,\bm a \in \mathcal F}  \int_{h_f-\epsilon}^{h_f} dh\\& =V(\bm y)+\min_{\bm a\in \R^2 ,\bm \nu\in \R^2}\left(-\nabla V(\bm y)\cdot \bm \nu+1\right)\epsilon+ O(\epsilon^2).
    \label{eq:HJB_deriv_mrap}
\end{align}
In the last step, the minimization is subject to the constraints\\
\begin{equation}
\label{eq:constraints_nu_a}
    \bm \nu\in \left\{\lambda \Tilde g(\bm y,\Bar z_0+h_f,-\bm a(h_f))\bm{\hat a}(h_f) \text{ s.t. }\lambda\in[0,1]\right\},\quad \bar z_0+h_f\geq h_\text{min}(\bm x(h_f)), \quad \bm a\in -\mathcal A_{gs}(\bm r(h))\,.
\end{equation}
From now on all the minima over $\bm a,\bm \nu$ are understood to be subject to these constraints. Since \eqref{eq:HJB_deriv_mrap} must hold for any small $\epsilon$, the term proportional to $\epsilon$ must vanish therefore giving
\begin{equation}
    \label{eq:HJB_mrap_implicit}
    \min_{\bm a,\bm \nu} (-\nabla V(\bm y))\cdot \bm \nu =-1
\end{equation}
For a given $\bm a$, let $C(\bm a)\subseteq \R^2$ be the set of values of $\bm \nu$ that satisfy the first two constraints in \eqref{eq:constraints_nu_a}. 
We now take the minimum with respect to $\bm \nu$ explicitly.
\begin{align}
    &\min_{\bm a,\bm \nu} (-\nabla V(\bm y))\cdot \bm \nu = \min_{\bm a\in -\mathcal A_{gs}(\bm r(h))}\min_{\nu\in C(\bm a)} (-\nabla V(\bm y))\cdot \bm \nu=\min_{\bm a\in \mathcal -A_{gs}(\bm r(h))}(-\nabla V(\bm y))\cdot \bm{\hat a}\max_{\bm \nu \in C(\bm a)} \norm{\bm \nu}
\end{align}
  In the second step we supposed that $\nabla V\cdot\bm{\hat a}> 0$, this is true when the minimum over $\bm a$ is taken. We also used that $\bm\nu=\bm{\hat a}\norm{\bm \nu}$.
\begin{align}
    \max_{\bm \nu \in C(\bm a)} \norm{\bm \nu}=
    \begin{cases}
    \left[\max\left(1/\Tilde g(\bm y,\Bar z_0+h_f,-\bm a), \;\nabla h_\text{min}(\bm y)\cdot \bm{\hat a}\right)\right]^{-1} & \text{if } V(\bm y)=h_\text{min}(\bm y) \\
    \Tilde g(\bm y,\Bar z_0+h_f,-\bm a), & \text{if } V(\bm y)>h_\text{min}(\bm y) 
  \end{cases}
  \end{align}

 The by-case solution is needed to appropriately enforce that the aircraft does not fall below $h_\text{min}$.
  The second case is the simplest. If the aircraft is above $h_\text{min}$ then it shall continue climbing with steepness $1/\Tilde g$. 
  If instead the aircraft is at $h_\text{min}$, then continuing the climb with a steepness of $1/\Tilde g$ might result in falling below $h_\text{min}$ in case the terrain rises more steeply. 
  When going in direction $\bm{\hat a}$ the terrain steepness is given by $\nabla h_\text{min}(\bm y)\cdot \bm{\hat a}$. Since the steepness of the path is $1/\norm{d\bm x/dh}=1/\norm{\bm \nu}$, we have that $1/\norm{\bm \nu}\geq \max\left(1/\Tilde g,\nabla h_\text{min}(\bm y)\cdot \bm{\hat a}\right)$.
  
  As we did in the case of GRRP we now split the maximum over $\bm a$ into a maximization over $\bm{\hat a}$ and $\norm{\bm a}$. Doing so and taking the max over $\norm{\bm{a}}$ explicitly we get
  \begin{equation}
      \max_{\norm{\bm a}} \max_{\bm \nu \in C(\norm{\bm a}\bm{\hat a})} \norm{\bm \nu} =\begin{cases}
    \left[\max\left(1/g(\bm y,\Bar z_0+h_f,-\bm{\hat a}), \;\nabla h_\text{min}(\bm y)\cdot \bm{\hat a}\right)\right]^{-1} & \text{if } V(\bm y)=h_\text{min}(\bm y) \\
     g(\bm y,\Bar z_0+h_f,-\bm{\hat a}), & \text{if } V(\bm y)>h_\text{min}(\bm y)\,,
  \end{cases}
  \end{equation}
  with the definition $g(\bm y, z, \bm{\hat a})\coloneqq \max_{\substack{\norm{\bm a}\in\R_+ \\\bm a\in-\mathcal{A}_{gs}((\bm y,z)) }}\Tilde g(\bm y,z,\bm{\hat a}\norm{\bm a})$.

  Plugging this back into \eqref{eq:HJB_mrap_implicit}, using that $h_f=V(\bm y)$, and that $\min_x f(x)=-\max_x -f(x)$, we get the final form of the HJB equations for MRAP, presented in the main text in \eqref{eq:HJB_mrap}:
  \begin{align}
    \label{eq:HJB_mrap_app}
      &\begin{cases}
    \max_{\bm{\hat a}\in S^1}\nabla V(\bm y)\cdot \bm{\hat a}\,\left[\max\left(1/g(\bm y,\Bar z_0+h_f,-\bm{\hat a}), \;\nabla h_\text{min}(\bm y)\cdot \bm{\hat a}\right)\right]^{-1}=1 & \text{if } V(\bm y)=h_\text{min}(\bm y) \\
     \max_{\bm{\hat a}\in S^1}\nabla V(\bm y)\cdot \bm{\hat a}\,g(\bm y, V(\bm y),-\bm{\hat a})=1, & \text{if } V(\bm y)>h_\text{min}(\bm y)\,,
  \end{cases}\\
  &V(\bm x_a)=h_\text{min}(\bm x_a).
  \end{align}
  In the equation we replaced $\Bar z_0$ with $h_\text{min}(\bm x_a)$, since at optimality this is the minimal altitude over $\bm x_a$.
  \subsection{Simplifications of the HJB equations for MRAP}
  \label{sec:simplifications_hjb_mrap}
 In this section, we aim to derive simplified forms of \eqref{eq:HJB_mrap_app}, with the goal of better understanding it.
 We start from the one dimensional case, without wind. When $V>h_\text{min}$ the equation is $|dV/dy(y)|>1/g$. We now consider the case $V=h_\text{min}$.
 \begin{itemize}
     \item If $y\geq x_a$, then $\hat a=1$ and $V$ obeys $\frac{dV}{dy}(y)=\max(1/g,\frac{dh_\text{min}}{dy}(y))$.
     \item If $y<x_a$, then $\hat a=-1$ and $V$ obeys $\frac{dV}{dy}(y)=-\max(1/g,-\frac{dh_\text{min}}{dy}(y))$
 \end{itemize}
Putting the two cases together yields the final form of the equations in one dimension:
  \begin{equation}
      \left|\frac{dV}{dy}(y)\right|=\begin{cases} 1/g  & \text{if } V(y)>h_\text{min}(y)\\
      \max\left(\left|\frac{dh_\text{min}}{dy}(y)\right|,1/g\right) & \text{if } V(y)=h_\text{min}(y).
      \end{cases}
  \end{equation}
  Figure \ref{fig:mrap_1d_sketch} shows the behavior of $V$ in a one-dimensional setting. 
  From the figure we observe that $V$ grows with steepness $1/g$ until it gets to the same height of $h_\text{min}$. Then it starts tracking $h_\text{min}$ as long as $dh_\text{min}(y)/dy>1/g$.  This captures a central feature of MRAP. In the absence of obstacles and with constant wind, $V$ forms a cone around $\bm x_a$ with steepness $1/g$.
  From this figure, we notice an important feature of the solution: even if $h_\text{min}$ is differentiable, $V$ is not. In the points where we transition from the cone solution to tracking the terrain the derivative of $V$ is discontinuous. 
  \begin{figure}
      \centering
      \includegraphics[width=0.8\textwidth]{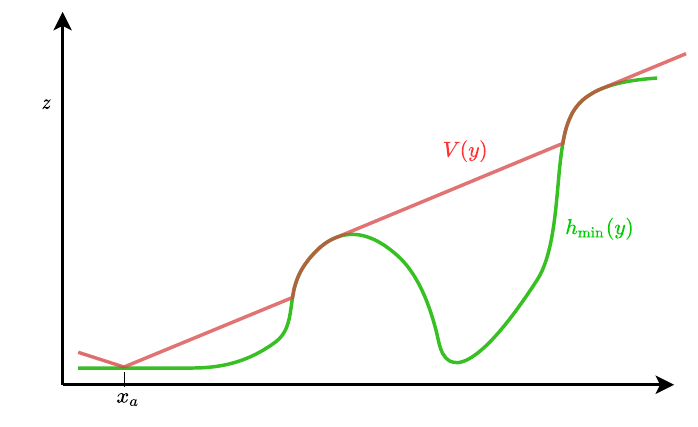}
      \caption{Depiction of the one dimensional solution to \eqref{eq:HJB_mrap_app}. When $V(y)>h_\text{min}(y)$, $|dV/dy|=1/g$. Instead when $V(y)=h_\text{min}(y)$, $|dV/dy|=\max (1/g,|dh_\text{min}/dy|)$. For the image $g=2.4$ was used.}
      \label{fig:mrap_1d_sketch}
  \end{figure}

 We shall now consider the 2d case in the absence of wind.
 Looking at equation \eqref{eq:HJB_mrap_app} let's first consider the case  $V(\bm y)>h_\text{min}(\bm y)$ (no interaction with obstacles). In this case, the aircraft climbs with steepness $1/g$. In the absence of wind, $g$ is constant with respect to $\bm{\hat a }$ and $V(\bm y)=h_\text{min}(\bm x_a)+\frac{1}{g}\norm{\bm y-\bm x_a}$.

  Now let's introduce the interaction with the terrain.
  If the ground rises more steeply than $1/g$ at some point the cone solution $V(\bm y)=h_\text{min}(\bm x_a)+\frac{1}{g}\norm{\bm y-\bm x_a}$ will fall below $h_\text{min}$, hence becoming invalid. The first case of \eqref{eq:HJB_mrap_app} tells us how $V$ evolves when $V(\bm y)=h_\text{min}(\bm y)$. The equation governing $V$  in this case is 
  \begin{equation}
  \label{eq:mrap_HJB_windless_interm}
      \max_{\bm{\hat a}\in S^1}\frac{\nabla V(\bm y)\cdot \bm{\hat a}}{\max\left( 1/g, \;\nabla h_\text{min}(\bm y)\cdot \bm{\hat a}\right)}=1.
  \end{equation}
  Now suppose additionally that $\norm{\nabla h_\text{min}(\bm y)}>1/g$ (terrain steeper than the glide cone). We show that setting $\nabla V(\bm y)=\nabla h_\text{min}(\bm y)$ solves \eqref{eq:mrap_HJB_windless_interm}.
  By plugging this ansatz in the equation we get
  \begin{equation}
  \label{eq:mrap_HJB_windless_interm2}
      \max_{\bm{\hat a}\in S^1}\frac{\nabla h_\text{min}(\bm y)\cdot \bm{\hat a}}{\max\left( 1/g, \;\nabla h_\text{min}(\bm y)\cdot \bm{\hat a}\right)}=1.
  \end{equation}
  which is maximized by $\bm{\hat a}^\star=\nabla h_\text{min}(\bm y)/\norm{\nabla h_\text{min}(\bm y)}$ and the maximum value is one, thus satisfying the equation.\footnote{Notice that there might be several $\bm{\hat a}$ attaining the maximum in \eqref{eq:mrap_HJB_windless_interm2}.} Once again we see that when the terrain is steeper than $1/g$, then $V$ starts tracking $h_\text{min}$.
We can write the complete HJB equation for MRAP in the zero wind case as 
\begin{equation}
 \label{eq:mrap_windless}
       \norm{\nabla V(\bm y)}=\begin{cases} 1/g  & \text{if } V(y)>h_\text{min}(y)\\
     \max\left(1/g, \norm{\nabla h_\text{min}(\bm y)}\right) & \text{if } V(y)=h_\text{min}(y).
      \end{cases}
  \end{equation}
This is the equation solved in practice by the Glikonal-M algorithm.
Put simply, the minimal return altitude is determined by the most stringent of two constraints: the glider needs to have enough altitude to glide back to the airport and it must pass above obstacles.
\section{Analytical solution of GRRP and MRAP}
\label{app:analytic_grrp_mrap}
In the case of uniform wind and no elevation (i.e. $h_\text{min}(\bm x)=-\infty \; \forall \bm x$ ) we can solve both the GRRP and the MRAP equations exactly.
We show this in the case of GRRP, but the solution for MRAP is identical.
Suppose $\bm W$ is the wind vector in the whole space. Let $z_0\in \R$ be the initial altitude of the glider and $\bm x_0$ its initial position. Since the wind is uniform, the function $g$ depends only on $\bm{\hat a}$. Equations \eqref{eq:HJB_S1} and \eqref{eq:constraints_HJB_S1} then become
\begin{align}
    \label{eq:HJB_S1_unif_wind}
    &\max_{\bm{\hat a}\in S^1} \nabla U(\bm y)\cdot \bm{\hat a} g\left(\bm{\hat a}\right) =1\\
    \label{eq:constraints_HJB_S1_no_terrain}
    &U(\bm{x_0})=0
\end{align}

The function $U(\bm y)=\norm{\bm y-\bm x_0}/g\left(\frac{\bm y-\bm x_0}{\norm{\bm y-\bm x_0}}\right)$ solves the equations above.
We prove this by inspection. Name $\bm r=\bm y-\bm x_0$ and $\bm{\hat r}=\bm r/\norm{\bm r}$.
\begin{equation}
    (\nabla U)_j(\bm y)= \frac{\hat r_j}{g(\bm{\hat r})}- \frac{1}{g(\bm{\hat r})^2}\sum_{i=1}^2\frac{\partial g}{\partial \hat r_i}(\bm{\hat r})\frac{\partial \hat r_i}{\partial r_j}= \frac{\hat r_j}{g(\bm{\hat r})}-\frac{1}{g(\bm{\hat r})^2\norm{\bm r}}\sum_{i=1}^2 \frac{\partial g}{\partial \hat r_i}(\bm{\hat r})\left(\delta_{ij}-\hat r_i\hat r_j\right).
\end{equation}
Now we compute $ \nabla U(\bm y)\cdot \bm{\hat a}=\nabla U(\bm y)\cdot \bm{\hat r}$.
\begin{equation}
    \nabla U(\bm y)\cdot \bm{\hat r}=  \frac{1}{g(\bm{\hat r})}-\frac{1}{g(\bm{\hat r})^2\norm{\bm r}}\sum_{i,j=1}^2 \frac{\partial g}{\partial \hat r_i}(\bm{\hat r})\left(\delta_{ij}\hat r_j-\hat r_i\hat r_j^2\right)=\frac{1}{g(\bm{\hat r})},
\end{equation}
To verify that this is the maximum one must then verify that the second derivative of $\nabla U(\bm y)\cdot \bm{\hat a}$ is positive.

Under the same assumptions it is possible to obtain an analytical expression for $V$, solving \eqref{eq:HJB_mrap} in MRAP. 
Its form is 
\begin{equation}
    V(\bm y)=\norm{\bm y-\bm x_0}/g\left(\frac{\bm x_0-\bm y}{\norm{\bm y-\bm x_0}}\right),
\end{equation}
Where $\bm x_0$ this time indicates the airfield's position.
The only difference with $U$ is the change of sign in the argument of $g$.
\section{Algorithms for GRRP}
\label{app:algos_GRRP}
In this section, we present an algorithm to solve GRRP in the absence of wind. This pseudocode does not reflect the actual implementation of Glikonal-G in the GitHub repository. In particular Algorithm \ref{alg:Glikonal_G_nowind} is based on the FMM while the repository implementation is based on OUMs. We still find it valuable to include this pseudocode as it shows how the algorithm works in a simplified setting, namely the case of zero wind.
The OUM based implementation can account for an arbitrary wind vector field, but it is also more complex. 
The computational complexity of both \ref{alg:Glikonal_G_nowind} and its OUM version is $O(n\log n)$ where $n$ is the number of nodes in the graph.

\begin{algorithm}[H]
\caption{Version of Glikonal-G without wind} 
\label{alg:Glikonal_G_nowind}
\begin{algorithmic}[1]
\State\textbf{Input: } Elevation profile $E=[E[0],\dots,E[n-1]]$, aircraft site $i_0$, grid spacing $h$, initial altitude of the aircraft $z_0$, glide ratio $g$, adjacency list of the grid $N=[N[0],\dots,N[n-1]]$. The list $N[i]$ stores the neighbours of $i$ in the order up, right, down, left. If node $i$ is on the boundary and therefore has fewer than four neighbours, the corresponding entries in $N[i]$ are set to $-1$.
\State\textbf{Output:} The lost altitude function $\{U_{i}\}$. If a node $i$ is not reachable then $U[i]=+\infty$.
\State known=$[\;]$ \Comment{List of nodes for which we know the minimal altitude}
\State considered=$[i_0]$ \Comment{List of nodes reached by the front}
\For{$i=-1,0,\dots,n-1$} \Comment{$U[-1]=+\infty$ is the value of the fictitious neighbour(s) for nodes on the boundary}
\State $U[i]\gets+\infty$
\EndFor
\State $U[i_0]\gets 0$
\While{considered is not empty}
\State $i\gets\arg\min_{j\in \text{considered}}U[j]$   
\State known.append($i$)
\State considered.remove($i$)
\For{$j\in N[i] \setminus \{-1\}$}
\If{$j\not\in\text{known}$}
\State $\Tilde U \gets \text{EikonalUpdate}(U,j,h,g,N)$ \Comment{Proposing new value for $U[j]$}
\If{$z_0-\Tilde U\geq E[j]$ \textbf{and} $\Tilde U<U[j]$}
\State $U[j]\gets\Tilde U$
\If{$j\not\in\text{considered}$}
\State considered.append($j$)
\EndIf
\EndIf
\EndIf
\EndFor
\EndWhile
\State \textbf{return} $U$
\end{algorithmic}
\end{algorithm}
In this algorithm, the 2d space is discretized into a grid with $n$ nodes. $h$ is the spacing between adjacent nodes on the grid. The edges in the grid are represented through an adjacency list '$N$'.
The list $E$ contains the function $h_\text{min}$ evaluated on the grid sites.

The algorithms works by dividing the nodes into three classes: far, considered, and accepted. Far nodes are those that have not been yet reached by the expanding front. Accepted ones are those for which $U$ is known. Considered nodes are nodes on the front. Initially, only the node corresponding to the initial position of the glider is in the considered set.
At each step in the while loop, we take the considered node with the lowest value of $U$ and move it into the accepted nodes. Then we compute a tentative value $\Tilde U$ for its neighbours. If $z_0-\Tilde U>h_\text{min}$ then the neighbour is added to the considered nodes (if it wasn't already among the considered nodes). The process halts when there are no more considered nodes.

To find the minimum in $O(\log n)$ time in line 10, a heap should be used to store the considered nodes. 
In the case of no wind, the vector field $\bm{\hat a }^\star_G$, defined in \eqref{eq:opt_vect_field_grrp} is equal to $\nabla U$.
Below we give the implementation of the function EikonalUpdate
\begin{algorithm}[H]
\caption{EikonalUpdate} 
\label{alg:EikonalUpdate}
\begin{algorithmic}[1]
\State\textbf{Input: } function $U$ (representing the negative altitude in GRRP or the altitude in MRAP), position to update $j$, grid spacing $h$, glide ratio $g$, adjacency list of the grid $N=[N[0],\dots,N[n-1]]$.
\State\textbf{Output:} $\Tilde U$, the updated altitude at position $i$.
\State $U_x\gets \min(U[N[j][1]],U[N[j][3]])$ \Comment{Taking the minimum between the right and left neighbours}
\State $U_y\gets \min(U[N[j][0]],U[N[j][2]])$ \Comment{Taking the minimum between the up and down neighbours}
\If{$|U_x-U_y|\leq h/g$}
\State $\Tilde U\gets \frac{1}{2}\left[U_x+U_y+\sqrt{2(h/g)^2-(U_x-U_y)^2}\right]$ \Comment{Solution of the quadratic equation $(U_x-\Tilde U)^2+(U_y-\Tilde U)^2=(h/g)^2$}
\Else
\State $\Tilde U\gets \min\left(U_x,U_y\right)+h/g$ \Comment{When the quadratic does not admit a solution use this one-dimensional update}
\EndIf

\State \textbf{return} $\Tilde U$
\end{algorithmic}
\end{algorithm}

\section{Altitude loss in turns}
\label{app:altitude_loss_turns}
The modeling of optimal control problems in both GRRP and MRAP presents some differences from the real behavior of a real aircraft.
\begin{enumerate}
    \item In our model, the glider can make arbitrarily sharp turns (e.g. instantly reverse the glider's course).
    \item At fixed airspeed the sink rate increases when turning in a real glider. This is because in a turn, the lift increases and so does the induced drag. Our model does not capture this, since the sink rate does not depend on the curvature of the path.
\end{enumerate}
Both problems arise because the ground speed is not required to be a continuous, or Lipschitz function.
The effect of the first problem is that the glider will turn sharply, for example in figure \ref{fig:grrp_test_2saddles}.

The second problem instead leads to underestimating the lost altitude, as the effect of turns is neglected. We now aim to characterize the importance of this phenomenon and derive when the approximations made by the dynamical model break down. Then we present some possible algorithmic remedies, to take into account at least the first turn (i.e. the one from the initial to the desired heading). 
\subsection{Modeling the altitude loss in a turn}
Many works \cite{di2016optimizing, chen2024gliding} that aim to compute the gliding range in the absence of obstacles assume that the glider will take an initial turn with constant bank angle from the initial heading towards the final heading. Once the turn is completed the glider continues in a straight line. 

This type of concatenated trajectory is depicted in figure \ref{fig:concat_traj}.

In this case, the glider's initial position is $(0,0)$ and its heading is indicated by the black arrow. $\bm x$ is the final position. The red line represents the optimal trajectory under our model of the dynamics. The concatenation of the blue and green lines is instead a more reasonable trajectory that can actually be flown. 

As an effect of having to turn towards $\bm x$, the gliding reachable region extends further in front of the aircraft than behind it. This is also evident from the results of \cite{di2016optimizing, chen2024gliding}.

Our goal is to compare the altitude loss on the red trajectory with that of the green-blue one, to understand in which cases the initial turn contributes significantly to the total altitude loss.
We assume that the glider's trajectory is composed of an initial circular arc, of fixed radius $R$, followed by a straight line. We indicate with $g_R$ the glide ratio during the turn, and with $g_\infty$ the glide ratio during straight flight. indicating with $\Delta H_t, \Delta H$ respectively the altitude losses on the green-blue and red trajectories.
\begin{align}
    &\Delta H=\norm{\bm x}/g_\infty\\&
    \Delta H_t=R\phi/g_R+\norm{\bm r_2}/g_\infty
\end{align}
Their difference is $\Delta H_t-\Delta H=\frac{1}{g_\infty}\left(\norm{\bm r_2}-\norm{\bm x}\right)+R\phi/g_R$.

We notice that $\norm{\bm x}-\norm{\bm r_2}\leq 2R$, giving

\begin{equation}
    \Delta H_t-\Delta H\leq R\left(\frac{\phi}{g_R}+\frac{2}{g_\infty}\right)
\end{equation}
We see that the larger the turn radius of the aircraft, the more potential discrepancy between the two altitude losses. The lower glide ratio $g_R$ also plays a role. Finally the larger the angle between the initial and final heading, the larger the discrepancy.

Ultimately, the parameter controlling the altitude loss is $\phi R/g_R$. The larger this number the more altitude is lost when turning onto a different heading. 

At this point, it is natural to wonder if there is a way to compensate for this difference in altitude loss.
The simplest option is to use a more conservative glide ratio (i.e., a smaller one) for the computation of $\Delta H$. Naming this $\Tilde g_\infty<g_\infty$, we have $\Delta H_t-\Delta H=\frac{\norm{\bm r_2}}{g_\infty}-\frac{\norm{\bm x}}{\Tilde g_\infty}+R\phi/g_R$. When $\norm{\bm x}$ is sufficiently large, this quantity is negative. This option however is not satisfactory when $\norm{\bm x}$ is too small.

To overcome this problem we can use a different approach.
As we presented it, the Glikonal-G algorithm starts expanding the front from a single node. In general, however, Glikonal-G can also take as input a partial solution, i.e., a function $U$ whose values are defined only on a certain region of space. Then Glikonal-G expands the solution to the whole space.

One can then use the trajectories composed by a turn plus a straight line to compute $U$ in a region around the aircraft (one must assume that in this region the aircraft does not collide with obstacles). Then, running Glikonal-G the solution is extended to the entire domain.
This allows to account approximately for the lost altitude during the initial turn. In case the aircraft makes additional turns to avoid terrain, it is not possible to correct the altitude loss.
\begin{figure}
    \centering
    \includegraphics[width=0.7\textwidth]{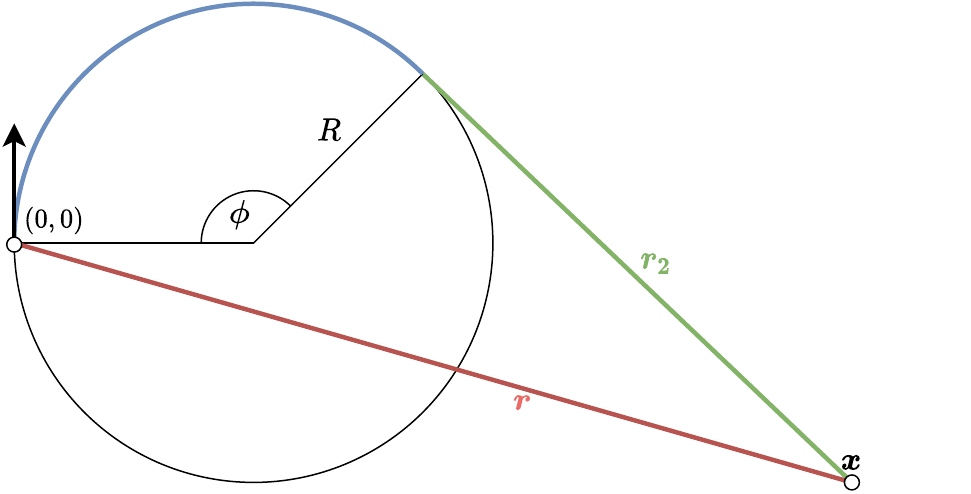}
    \caption{The glider's initial position is $\bm{x_0}$, and its final position is $\bm x$. The glider's initial heading is up. The red line represents the optimal trajectory going from $\bm x_0$ to $x$, under our model. The concatenation of the blue and green lines is instead a more realistic trajectory}
    \label{fig:concat_traj}
\end{figure}

\section{Algorithm for MRAP}
\label{app:algo_mrap}
In this appendix, we provide the pseudocode of the Glikonal-M algorithm. Glikonal-M works on a grid with identical horizontal and vertical spacings. This implementation is does not take wind into account and solves equation \eqref{eq:mrap_windless}. It is based on the Fast Marching Method. The following pseudocode reflects the algorithm published in the repository.
\begin{algorithm}[H]
\caption{Glikonal-M. Algorithm to solve MRAP in the absence of wind.} 
\label{alg:Glikonal_M_nowind}
\begin{algorithmic}[1]
\State\textbf{Input: } Elevation profile $E=[E[0],\dots,E[n-1]]$, airfield site $i_0$, grid spacing $h$, glide ratio $g$,  adjacency list of the grid $N=[N[0],\dots,N[n-1]]$. The list $N[i]$ stores the neighbours of $i$ in the order up, right, down, left. If node $i$ is on the boundary and therefore has fewer than four neighbours, the corresponding entries in $N[i]$ are set to $-1$.
\State\textbf{Output:} The minimal return altitude function $\{V_{i}\}$.
\State known=$[\;]$ \Comment{List of nodes for which we know the minimal altitude}
\State considered=$[i_0]$ \Comment{List of nodes reached by the front}
\For{$i=-1,0,\dots,n-1$} \Comment{$V[-1]=\infty$ is the value of the fictitious neighbour for nodes on the boundary}
\State $V[i]\gets \infty$
\EndFor
\State $V[i_0]\gets E[i_0]$
\While{considered is not empty}
\State $i\gets\arg\min_{j\in \text{considered}}V[j]$   
\State known.append($i$)
\State considered.remove($i$)
\For{$j\in N[i] \setminus \{-1\}$}
\If{$j\not\in\text{known}$}
\State $V[j] \gets \min\left\{\max\left(\text{EikonalUpdate}(V,j,h,g,N), E[j]\right), V[j]\right\}$ \Comment{Updating the value of $V[j]$}
\If{$j\not\in\text{considered}$}
\State considered.append($j$)
\EndIf
\EndIf
\EndFor
\EndWhile
\State \textbf{return} $V$
\end{algorithmic}
\end{algorithm}
In this algorithm, the 2d space is discretized into a grid with $n$ nodes. $h$ is the spacing between adjacent nodes on the grid. The edges in the grid are represented through an adjacency list '$N$'.
The list $E$ contains the function $h_\text{min}$ evaluated on the grid sites.

The algorithms works by dividing the nodes into three classes: far, considered, and accepted. Far nodes are those that have not been yet reached by the expanding front. Accepted ones are those for which $V$ is known. Considered nodes are nodes on the front. Initially, only the node corresponding to the position of the airfield is in the considered set.
At each step in the while loop, we take the considered node with the lowest value of $V$ and move it into the accepted nodes. Then we 
 add its neighbours to the considered set (if they were not there already) and compute a tentative value of $V$ for each neighbour. 
 The tentative value is computed by taking the maximum between the elevation $E$ and the altitude needed to glide back to the airport in the absence of the obstacle.
 The process halts when there are no more considered nodes.

To find the minimum in $O(\log n)$ time in line 10, a heap should be used to store the considered nodes. 
In the case of no wind, the vector field $\bm{\hat a }^\star_M$, defined in section \ref{sec:mrap_main}, is equal to $\nabla U$.
The implementation of the function EikonalUpdate is in \ref{alg:EikonalUpdate}.

\section{Additional numerical experiments}
\label{app:additional_numerics}
In this appendix, we present an additional experimental setting for Glikonal-G and Glikonal-M. We consider specifically the case of an elevation profile including a mountain range with two passes. In both cases, the glide ratio in the absence of wind is one and the grid has size $101\times 101$, with a grid spacing of one.
\subsection{GRRP}
We assume that the glider travels at fixed airspeed of one. The wind has the functional form $\bm W(\bm y)=(0,(50-y_1)/100)$. The wind vector field is plotted on top of the elevation profile ($h_\text{min}$) in figure \ref{fig:grrp_mountain_range}(a). The initial altitude of the glider is $z_0=110$. Panel (b) depicts the contour lines of $z_0-U_G$ in white and the optimal trajectories in turquoise. The red shaded regions are unreachable in gliding flight. We observe that the glider is forced to pass the mountain range through one of the two passes.
In (c), a three dimensional view of the elevation profile and of the function $z_0-U_G$ is presented.
\begin{figure}
    \centering
    \includegraphics[width=\textwidth]{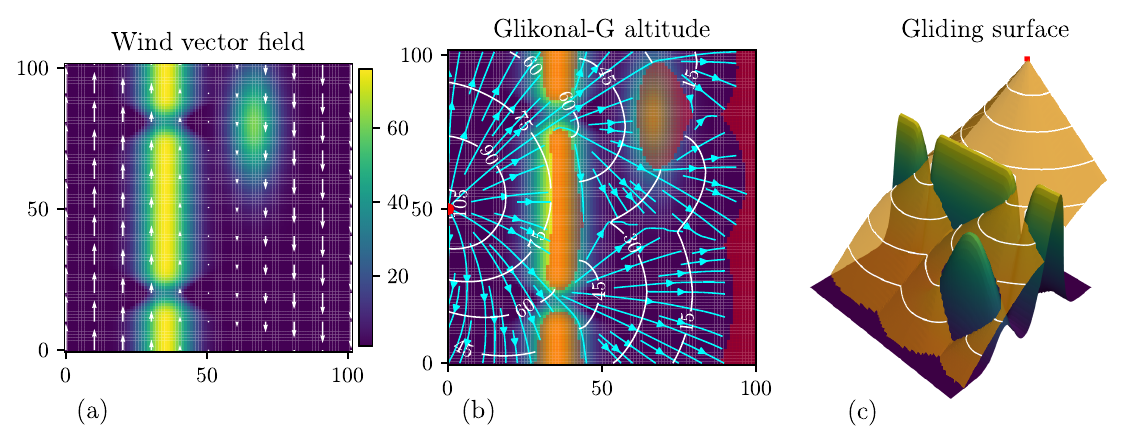}
    \caption{Glikonal-G's solution in the case of non-uniform wind and a mountain range with two passes. The red dot represents the glider's initial position. \textbf{(a)} Wind direction (constant with altitude) plotted on top of the elevation profile (displayed as a heatmap). \textbf{(b)} Contour lines of the function $z_0-U_G$ (in white) and optimal trajectories (in turquoise). The heatmap is the elevation profile, and the red-shaded regions are those not reachable in gliding flight. \textbf{(c)} 3d representation of the function $z_0-U_G$, in orange and of $h_\text{min}$.}
    \label{fig:grrp_mountain_range}
\end{figure}
\subsection{MRAP}
We test Glikonal-G on an elevation profile depicted in figure \ref{fig:mrap_mountain_range}(a) as a heatmap. The purple dot indicates the position of the airfield. In the same plot, the contour lines of $V$ are shown in white, and the feasible re-entry paths in turquoise. We remark that the re-entry trajectories tend to pass through one of the two saddles. In a small region just behind the middle of the mountain range, it is then convenient to pass the range directly, since reaching a saddle would lead to a greater loss of altitude. This region is well visible in panel (b), where it takes a triangular shape, with the base at the top of the mountain range.

\begin{figure}
    \centering
    \includegraphics[width=0.8\textwidth]{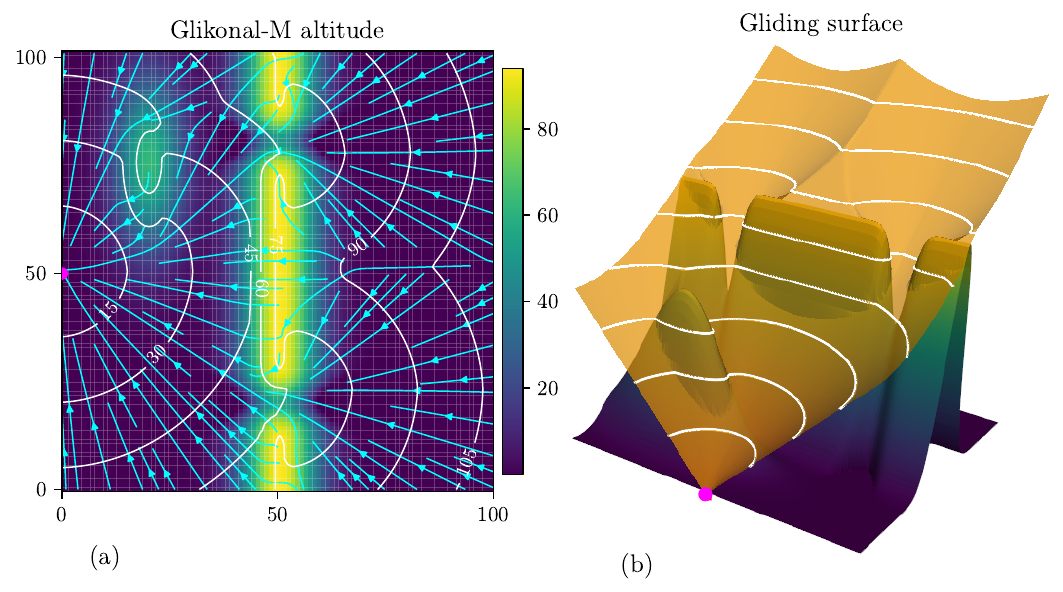}
    \caption{Glikonal-M's solution in the case of an elevation profile with a mountain range with two passes. The purple dot represents the airfield's position. \textbf{(a)} Heatmap of the elevation profile with contour lines of the function $V_G$ (in white) and re-entry trajectories (in turquoise). \textbf{(b)} 3d plot of the function $V_G$ in orange on top of the elevation profile.}
    \label{fig:mrap_mountain_range}
\end{figure}
\section{Error in Glikonal-M and Glikonal-G}
\label{app:error_algo}
In this appendix, we test the Glikonal-M and Glikonal-G algorithms in several settings where we can compute the true solution. The errors are due to the discretization of the PDE, which is numerically solved on a grid. We also verify that when the grid spacing decreases the error also does. For proofs of convergence of the discretized solution to the true one when the grid spacing approaches zero we refer the reader to \cite{sethian2003ordered}.

\subsection{Glikonal-G}
Name $U_G$ the approximation of $U$ outputted by Glikonal-G. The following experiments are all run on a $101\times101$ grid with a spacing of one. The glide ratio is also one in the absence of wind. When wind is present the glide ratio in direction $\bm{\hat a}$ is $g(\bm x, y,\bm{\hat a})=\norm{\bm{\hat a}+\bm W(\bm x,z)}$, corresponding to the fact that the glider moves in every direction with the same airspeed of one. The red dot indicates the initial position of the glider.

The first experiment is similar to the one of figure \ref{fig:flat_glikonal_g}. We set $h_\text{min}=0$ everywhere. The wind is taken to be uniform with strength $0.6$ and direction $225$°. The initial altitude is $100$.
The results are displayed in figure \ref{fig:grrp_test_flat}. In panel (a) we plot the function $z_0-U_G$.
In (b) the function $(U_G-U)/U$ is plotted. We remark that the relative error barely exceeds 2\%. The total error $U_G-U$ is plotted in panel (c). The total error increases when moving away from the initial position while the relative one decreases. 
\begin{figure}[h]
    \centering
    \includegraphics[width=0.8\textwidth]{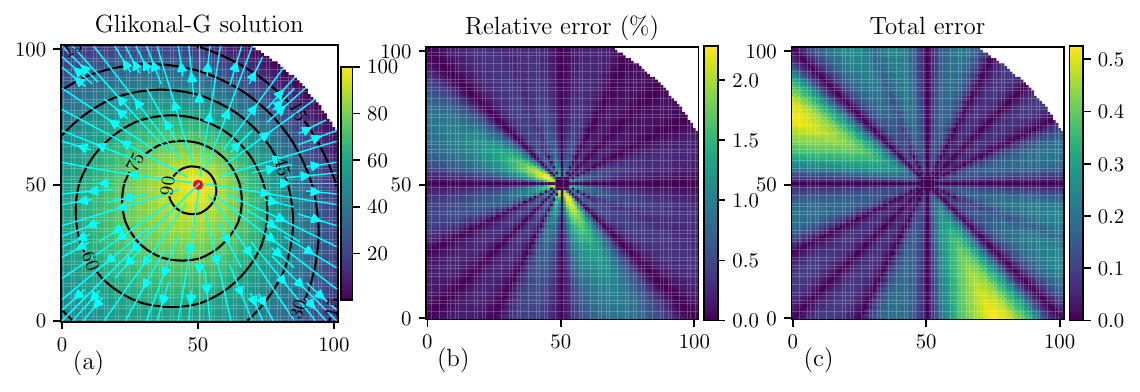}
    \caption{Test of Glikonal-G on flat terrain. \textbf{(a)} Heatmap and contour lines of the function $z_0-U_G$. The turquoise lines are the optimal trajectories and the red dot is the initial position of the glider. \textbf{(b)} Relative error, $(U_G-U)/U$. \textbf{(c)} Total error $U_G-U$.}
    \label{fig:grrp_test_flat}
\end{figure}

In the second experiment, we place an infinitely high barrier in the middle of the domain. We leave two small openings where $h_\text{min}=0$ respectively in the upper and lower portion of the barrier. Apart from the barrier we set $h_\text{min}=0$ everywhere else. The barrier is depicted as a white vertical line in figure \ref{fig:grrp_test_2saddles}. The initial altitude is 120. The wind is uniform with strength $0.4$ and pointing north. In (a) we plot $z_0-U_G$. The panels (b,c) instead show the relative and total error, defined as above. The total error reaches its maximum value on the boundary of the GRR. In figure \ref{fig:test_mrap_2saddles_h025} we give a numerical demonstration of the fact that when the grid spacing goes to zero, $U_G$ approaches $U$. We used a grid with size $404\times 404$ and spacing 0.25. By comparing panels (b,c) with those of the previous figure we observe that both the relative and total error are greatly reduced. On the boundary of the GRR some small errors remain.
\begin{figure}[h]
    \centering
    \includegraphics[width=0.8\textwidth]{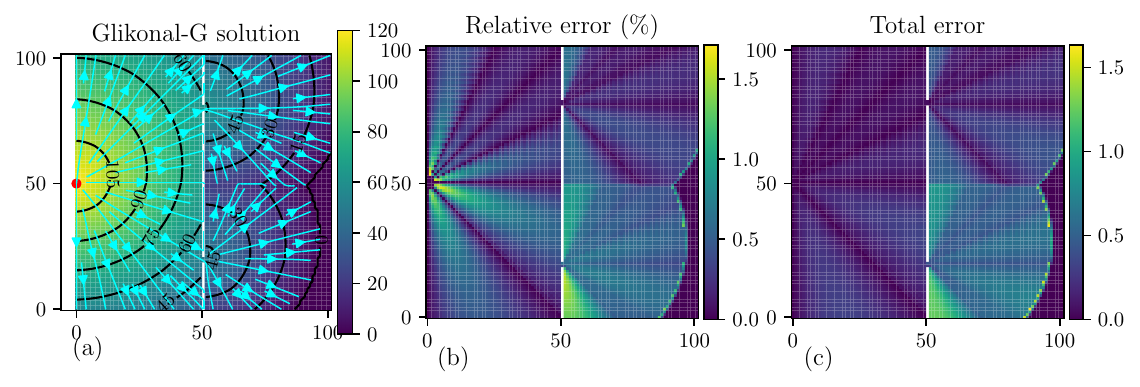}
    \caption{Test of Glikonal-G on elevation profile with an infinite barrier with two passes. The infinite barrier is portrayed as a vertical white line, with two openings representing the passes. \textbf{(a)} Heatmap and contour lines of the function $z_0-U_G$. The turquoise lines are the optimal trajectories and the red dot is the initial position of the glider. \textbf{(b)} Relative error, $(U_G-U)/U$. \textbf{(c)} Total error $U_G-U$.}
    \label{fig:grrp_test_2saddles}
\end{figure}
\begin{figure}[H]
    \centering
    \includegraphics[width=0.8\textwidth]{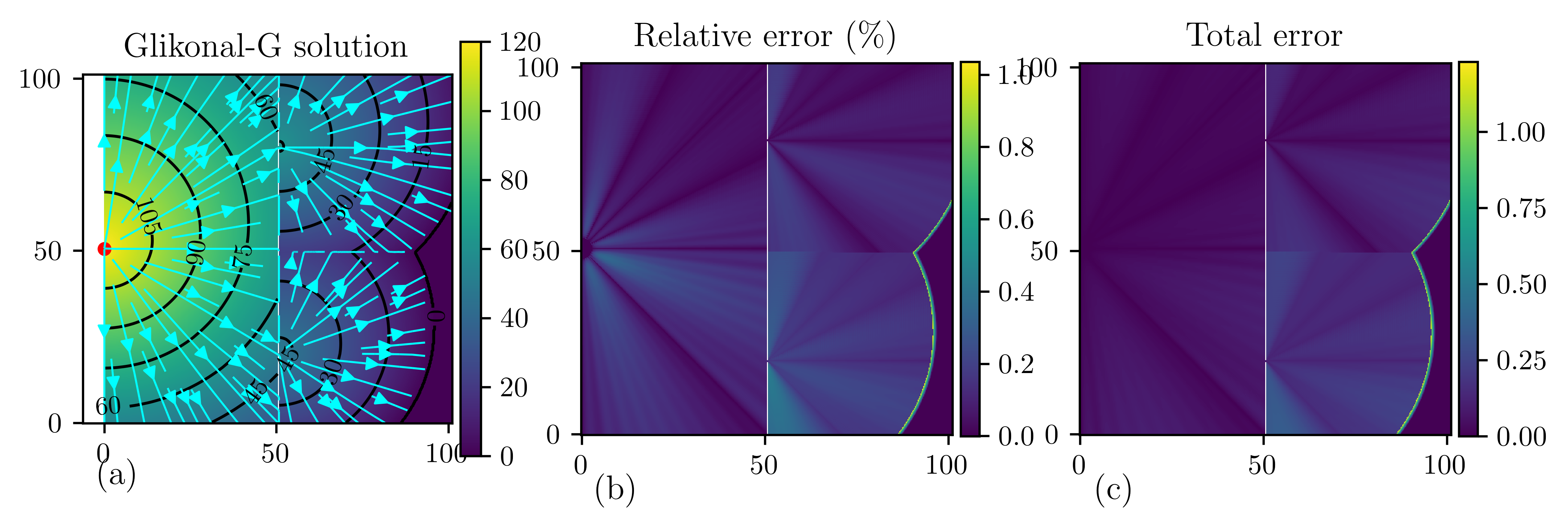}
    \caption{Test of Glikonal-G on elevation profile with an infinite barrier with two passes, same as in figure \ref{fig:grrp_test_2saddles}. The grid spacing is 0.25, one quarter of the grid spacing in the previous figure, and the grid size is $404\times 404$. \textbf{(a)} Heatmap and contour lines of the function $z_0-U_G$. The turquoise lines are the optimal trajectories and the red dot is the initial position of the glider. \textbf{(b)} Relative error, $(U_G-U)/U$. \textbf{(c)} Total error $U_G-U$.}
    \label{fig:enter-label}
\end{figure}

In the last test, we place a barrier in the same position as the previous experiment. This time however the barrier has a finite elevation of 45. This way the glider can pass the barrier if it approaches it directly, but cannot do so if it proceeds diagonally. The two white segments represent where the barrier stops the glider from passing. In the rest of space, we set $h_\text{min}=0$.
The wind has a strength of 0.3 and is directed towards the north. In panel (a) the function $z_0-U_G$ is plotted. The panels (b,c) instead show the relative and total error, defined as above. The maximum error is attained on the bottom right. In this case, the error arises from the fact that in the true solution, the glider manages to pass the barrier in a slightly larger (by one pixel) portion of it.

\begin{figure}[h]
    \centering
    \includegraphics[width=0.8\textwidth]{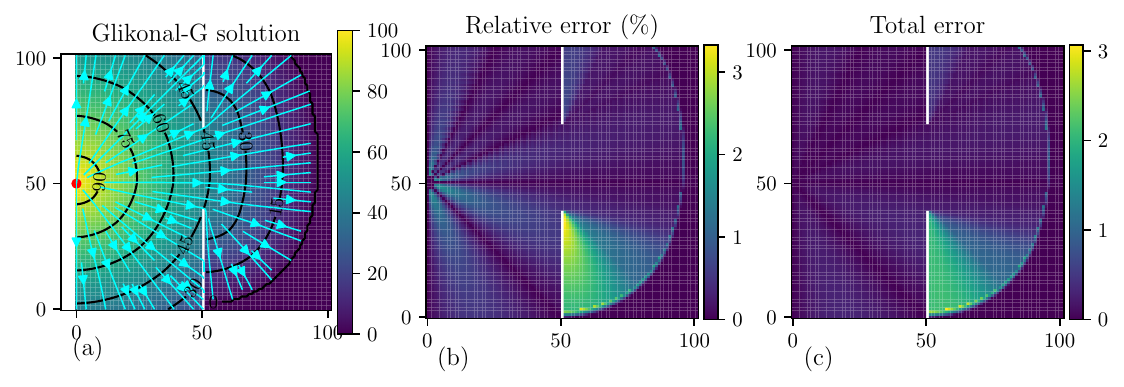}
    \caption{Test of Glikonal-G on elevation profile with a finite barrier of height 45. The glider has enough altitude to traverse the barrier in the middle but not on the sides. The two white segments represent where the glider fails to cross the barrier. \textbf{(a)} Heatmap and contour lines of the function $z_0-U_G$. The turquoise lines are the optimal trajectories and the red dot is the initial position of the glider. \textbf{(b)} Relative error, $(U_G-U)/U$. \textbf{(c)} Total error $U_G-U$.}
    \label{fig:grrp_test_range}
\end{figure}
In all experiments, the relative error of Glikonal-G is below 4\%. Moreover, all errors overestimate the altitude loss, leading to conservative estimates of the GRR.

\subsection{Glikonal-M}
We repeat the experiments of the last section for the case of Glikonal-G. The elevation profiles we consider are similar but not identical to those above. The purple dot indicates the position of the airfield.

Figure \ref{fig:test_mrap_flat} considers the case of a flat elevation profile, i.e., $h_\text{min}=0$ everywhere. 
Panel (a) shows the function $V_G$ via its contour lines and a heatmap. The turquoise lines are the feasible re-entry paths. In panels (b,c) the relative and total error, respectively $(V_G-V)/V$ and $V_G-V$ are plotted. 
We observe that the error is maximal along the $\pm 45$° (with respect to the x-axis) directions, while it is zero on the axes. Notice how the radial patterns of the error are different from those of Glikonal-G. While the presence of wind also plays a role, the difference mainly originates from the methods used to solve the HJB equations. 
In the case of Glikonal-G, the OUM method requires using a triangulated grid, which is built by adding an edge on one of the two diagonals of each square in the grid. Glikonal-M instead uses the FMM, which uses a square grid without any triangulation.

In the second test, we consider a flat elevation profile with a barrier of infinite height. The barrier has two openings where it can be passed ($h_\text{min}=0$ in correspondence of the openings). The barrier and the openings are shown in white in panels (b,c) of figure \ref{fig:test_mrap_2saddles}. In (a) the heatmap of $V_G$ and its contour lines are shown. The feasible return paths, in turquoise, can be seen to pass exclusively from the two openings. The relative error and total error are plotted in (b,c).
\begin{figure}[h]
    \centering
    \includegraphics[width=0.8\linewidth]{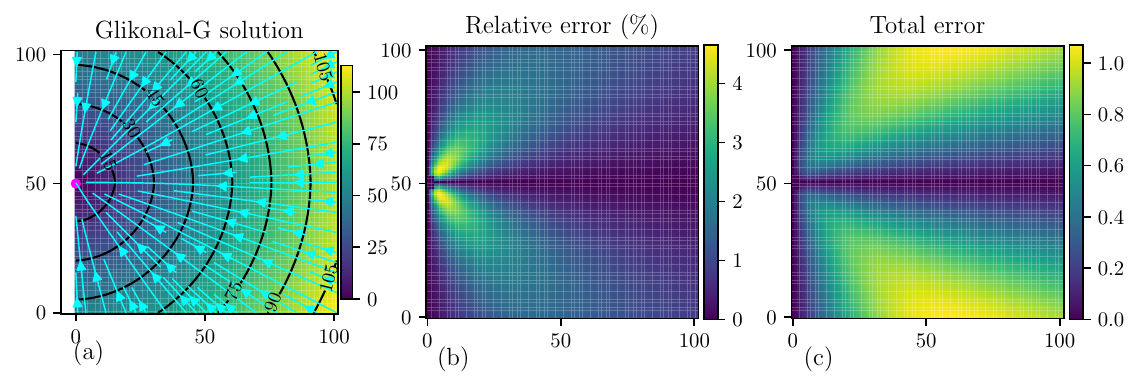}
    \caption{Test of Glikonal-M on flat terrain ($h_\text{min}=0$ everywhere). The purple dot indicates the position of the airfield. \textbf{(a)} Heatmap and contour lines of $V_G$, the solution outputted by Glikonal-M. The turquoise lines are the feasible re-entry trajectories. \textbf{(b)} Relative error $(V_G-V)/V$ in percent. \textbf{(c)} Total error $V_G-V$. }
    \label{fig:test_mrap_flat}
\end{figure}

\begin{figure}[h]
    \centering
    \includegraphics[width=0.8\linewidth]{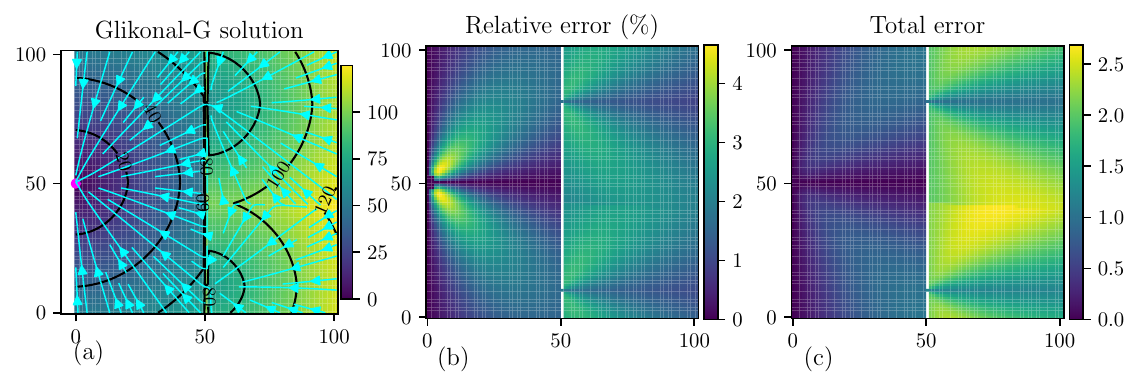}
    \caption{Test of Glikonal-M on elevation profile with an infinite barrier with two narrow openings. In (b,c) the barrier is visible as a white line broken in two points in correspondence of the openings. In the openings and outside the barrier $h_\text{min}=0$. The purple dot indicates the position of the airfield. \textbf{(a)} Heatmap and contour lines of $V_G$, the solution outputted by Glikonal-M. The turquoise lines are the feasible re-entry trajectories. \textbf{(b)} Relative error $(V_G-V)/V$ in percent. \textbf{(c)} Total error $V_G-V$.}
    \label{fig:test_mrap_2saddles}
\end{figure}
To show that $V_G$ approaches $V$ when the grid spacing approaches zero, we repeat the previous experiment on a grid with $404\times404$ nodes and grid spacing $0.25$. Figure \ref{fig:test_mrap_2saddles_h025} illustrates the results. Comparing the panels (a) in \ref{fig:test_mrap_2saddles_h025} and \ref{fig:test_mrap_2saddles} we observe no qualitative difference, however, the maximum relative error is reduced by about a factor 4, and so is the maximum total error.

\begin{figure}[h]
    \centering
    \includegraphics[width=0.8\textwidth]{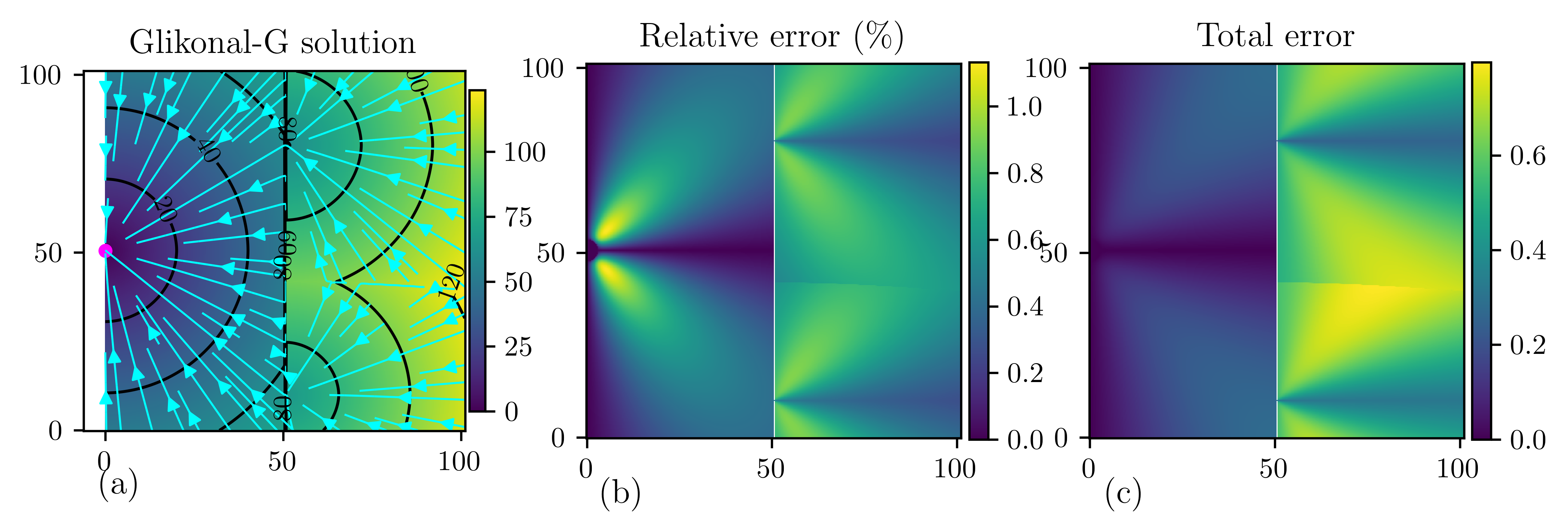}
    \caption{Test of Glikonal-M on elevation profile with an infinite barrier with two narrow openings. Same elevation profile as in figure \ref{fig:test_mrap_2saddles}. Grid spacing: $0.25$, grid size: $404\times404$. The purple dot indicates the position of the airfield. \textbf{(a)} Heatmap and contour lines of $V_G$, the solution outputted by Glikonal-M. The turquoise lines are the feasible re-entry trajectories. \textbf{(b)} Relative error $(V_G-V)/V$ in percent. \textbf{(c)} Total error $V_G-V$. }
    \label{fig:test_mrap_2saddles_h025}
\end{figure}
The final test, similar to \ref{fig:grrp_test_range} involves an elevation profile with a barrier in the same position as the previous test, but with a height of 60. Figure \ref{fig:test_mrap_barrier} (a) depicts the function $V_G$ and the feasible trajectories. Notice that if a glider is on the opposite side of the barrier, for a certain range of radials from the airfield, it must fly perpendicular to the barrier and will turn towards the airfield only after passing the barrier.
Panels (b,c) depict respectively the relative and total error of $V_G$.

\begin{figure}[h]
    \centering
    \includegraphics[width=0.8\linewidth]{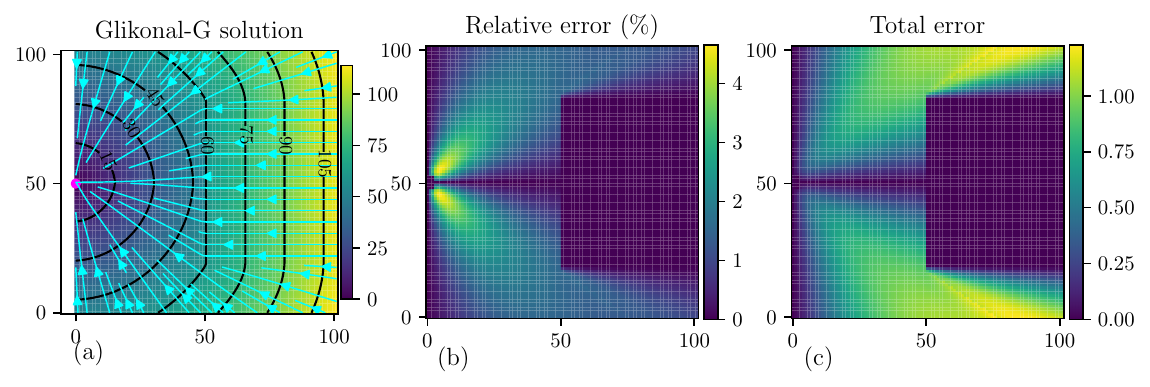}
    \caption{Test of Glikonal-M on elevation profile with a finite barrier. The barrier has a height of 60 and traverses the plot vertically. The x coordinate of the barrier is 50. The purple dot indicates the position of the airfield. \textbf{(a)} Heatmap and contour lines of $V_G$, the solution outputted by Glikonal-M. The turquoise lines are the feasible re-entry trajectories. \textbf{(b)} Relative error $(V_G-V)/V$ in percent. \textbf{(c)} Total error $V_G-V$.}
    \label{fig:test_mrap_barrier}
\end{figure}

Throughout these three tests, Glikonal-G's relative errors are always below  5\%. 
When an error is present, it is always on the conservative side, in the sense that $V_G>V$, hence the minimal re-entry altitude is overestimated.
\section{Trajectories in MRAP and GRRP}
\label{app:optimal_traj_mrap_grrp}
In this appendix we discuss the difference between the trajectories found when integrating the GRRP vector field $\bm{\hat a_G}$, \eqref{eq:opt_vect_field_grrp} and those found when integrating the MRAP vector field  $\bm{\hat a_M}$ introduced in \ref{sec:mrap_main}.
Let us first specify the setting. Let $\bm x_a\in\R^2$ and $\bm y\in\R^2$ be respectively the position of the airfield and of the glider. Suppose further that the wind is zero everywhere and that the altitude of the glider is $V(\bm y)$, the minimal altitude needed to reach the airfield, according to MRAP.
Solve GRRP with initial condition $(\bm y,V(\bm y))$ and compute the optimal trajectory (recall that trajectories in GRRP are guaranteed to lead to the least loss of altitude) $\gamma_G:\R_+\mapsto \R^2$ going from $\bm y$ to $\bm x_a$, by integrating the vector field $-\bm{\hat a}_G$ starting from $\bm x_a$.
Compare this to the trajectory $\gamma_M:\R_+\mapsto \R^2$ obtained by integrating $-\bm{\hat a_M}$ starting from $\bm y$. Both curves $\gamma_M$ and $\gamma_G$ connect $\bm x_a$ to $\bm y$, albeit going in opposite directions.

Do $\gamma_G$ and $\gamma_M$ follow the same path?
The answer in general is no. To prove this we will show a counterexample where the two curves are different.
Suppose the glide ratio is one and consider a 'staircase' elevation profile. Looking at the heatmap in figure \ref{fig:grrp_mrap_traj_diff} the elevation is 0 between 0 and 33 on the x-axis, it's 100 between 34 and 66, and it's 200 from 66 to 100. In this example, the solution of both GRRP and MRAP can be computed analytically. The turquoise lines are the curves $\gamma_M$ while the blue lines are the curves $\gamma_G$ (reversed in direction). We observe that the blue lines turn towards the airfield immediately after traversing the first 'step' of the staircase. Instead, the turquoise lines keep going straight and turn after traversing the second step. 
\begin{figure}
    \centering
    \includegraphics[width=0.4\textwidth]{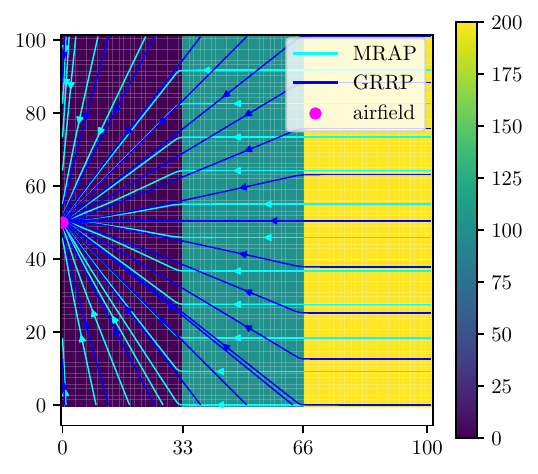}
    \caption{Trajectories of MRAP and GRRP. The heatmap represents a staircase elevation profile. The purple dot is the position of the airfield. The turquoise and blue lines are respectively the return trajectories outputted by MRAP and the optimal trajectories in GRRP.}
    \label{fig:grrp_mrap_traj_diff}
\end{figure}
Let us understand why this is the case. Let $y_1,y_2$ be respectively the x and y coordinates of $\bm y$.

The minimal altitude function $V$ in this case is 
\begin{equation}
V(\bm y)=
     \begin{cases}
    \norm{y} & \text{if } y_1\leq 33\\
   100+(y_1-33) & \text{if } 33<y_1\leq66\\
   200+(y_1-66) &\text{if } y_1>66
  \end{cases}
\end{equation}
As explained in  \ref{sec:simplifications_hjb_mrap}, the vector field in MRAP is $\bm{\hat a}_M=\nabla V(\bm y)/\norm{\nabla V(\bm y)}$, and is therefore parallel to the x-axis for $y_1>33$. 

Let us look at GRRP, and consider a point $\bm y$ with $y_1>66$. 
The curve $\gamma_G$ starting from $\bm y$ with altitude $V(\bm y)=200+(y_1-66)$ must initially go straight towards the point $(66,y_2)$, which will be reached at exactly an altitude of $200$. 
After that, however, $\gamma_G$ turns towards the airfield, as there are no more obstacles, and therefore proceeds in a straight line.

To conclude, the difference between $\gamma_G$ and $\gamma_M$ is that $\gamma_M$ always assumes that the aircraft is at the minimal altitude $V$, at each point in the trajectory, even when the terrain drops more sharply than the glide ratio. Instead $\gamma_G$ supposes that the glider flies with a fixed glide ratio.

Notice that in the absence of obstacles, the two trajectories coincide. This can be seen also from the analytical solutions presented in \ref{app:analytic_grrp_mrap}.
\end{document}